\documentclass{amsart}

\usepackage{amsmath}

\usepackage{amssymb}
\usepackage{amscd}

\numberwithin{equation}{section}

\newif\ifdraft\draftfalse

\font\sn = cmssi8 scaled \magstep0

\newcommand\name[1]{\label{#1}{\ifdraft{\sn [#1]}\else\ignorespaces\fi}}

\newcommand\eq[2]{{\ifdraft{\ \tt [#1]}\else\ignorespaces\fi}\begin{equation}\label{eq:#1}{#2}\end{equation}}

\newcommand {\equ}[1]     {\eqref{eq:#1}}

\newcommand{\R}{{\mathbb{R}}}

\newcommand{\Z}{{\mathbb{Z}}}

\newcommand{\N}{{\mathbb{N}}}

\newcommand {\ignore}[1]

\newcommand{\sm}{\smallsetminus}

\newcommand{\vre}{\varepsilon}

\newcommand{\su}{\mathbf{x}}

\usepackage{enumerate}
\usepackage{amssymb}
\usepackage{amsmath}
\usepackage{amscd}
\usepackage{amsthm}
\usepackage{amsfonts}
\usepackage{graphicx}
\usepackage[all]{xy}
\usepackage{verbatim}
\usepackage{hyperref}
\usepackage{epstopdf}
\usepackage[utf8]{inputenc}
\usepackage{bbm}
\usepackage{esint}
\usepackage{tikz}

\newtheorem{theorem}{Theorem}

\newtheorem{lemma}[theorem]{Lemma}

\theoremstyle{definition}

\theoremstyle{definition}

\theoremstyle{definition}

\theoremstyle{definition}
\theoremstyle{definition}
\theoremstyle{definition}
\numberwithin{theorem}{section}

\newcommand{\eps}{\varepsilon}

\newcommand{\cC}{\mathcal{C}}

\newcommand{\cH}{\mathcal{H}}

\newcommand{\cL}{\mathcal{L}}

\newcommand{\cS}{\mathcal{S}}

\newcommand{\bR}{\mathbb{R}}
\newcommand{\bZ}{\mathbb{Z}}

\newcommand{\bN}{\mathbb{N}}

\newcommand{\bS}{\mathbb{S}}

\newcommand{\SL}{\operatorname{SL}}

\renewcommand{\hom}{{G\backslash \Gamma}}

\newcommand{\SLR}[1][d]{\operatorname{SL}_{#1}(\bR)}

\newcommand\norm[1]{\left\|#1\right\|} % Norm \norm
\newcommand\set[1]{\left\{#1\right\}} % Set \set

\newcommand\on[1]{\operatorname{#1}} % Operatorname \on

\newcommand\abs[1]{\left|#1\right|}

\newcommand\inn[1]{\left\langle #1 \right\rangle}
\newcommand{\onto}{\xymatrix{\ar@{>>}[r]&}}
\newcommand{\da}[4]{\xymatrix{#1 \ar@<.5ex>[r]^{#2} \ar@<-.5ex>[r]_{#3} & #4}}

\newif\ifdraft\draftfalse
\font\sn = cmssi8 scaled \magstep0

\newcommand{\bwpar}[1]{\marginpar{{\color{olive}\small [BW] #1}}}

\newcommand{\supp}{\operatorname{Supp}}

\newcounter{constk}

\newcounter{consta}[section]
\renewcommand{\theconsta}{{\alpha_{\arabic{consta}}}}
\newcounter{constc}[section]

\newcounter{constt}[section]

\newcommand{\consta}{\refstepcounter{consta}\theconsta}

\begin{document}

\title{Effective Counting on Translation Surfaces}
\author{Amos Nevo}

\address{Department of Mathematics, Technion, Haifa, Israel } 
\email{anevo@technion.ac.il}

\author{Rene R\"uhr}
\address{Department of Mathematics, Technion, Haifa, Israel} 
\email{rener@campus.technion.ac.il}

\author{Barak Weiss}
\address{Department of Mathematics, Tel Aviv University, Tel Aviv, Israel} 
\email{barakw@post.tau.ac.il}

\maketitle

\begin{abstract}
We prove an effective version of a celebrated result of Eskin and
Masur: for any $\SL_2(\R)$-invariant locus $\cL$ of translation surfaces, there
exists $\kappa>0$, such that for almost every translation surface in
$\cL$, the number of saddle connections with holonomy vector of length
at most $T$, grows like $cT^2 + O(T^{2-\kappa})$. We also provide
effective 
versions of counting in sectors and in ellipses. 

\end{abstract}

\section{Introduction}
The main goal of this paper is the effectivization of a celebrated
result of Eskin and Masur \cite{EsMa} which we recall. A
translation surface $\su$ is a compact oriented surface equipped
with an atlas of planar charts, whose transition maps are translations, where the charts are
defined at every point of the surface except finitely many singular points at
which the planar structure completes to form a cone point of angle an
integer multiple of $2\pi$. Such structures arise in
many contexts in geometry, complex analysis and dynamics, and
have various equivalent definitions, see the surveys \cite{masur2002rational}, \cite{zorichflat} for more details. The collection of all translation surfaces
of a fixed genus, fixed 
number of singular points, and fixed cone angle at each singular point
is called a {\em stratum}, and has a natural structure of a linear 
orbifold. Furthermore each connected component of the subset of area one surfaces in a stratum is
the support of a natural smooth probability measure
which we will call {\em flat measure}.

A {\em saddle connection} on a
translation surface $\su$ is a segment connecting two singular
points which is linear in each planar chart and contains no singular
points in its interior. The {\em holonomy vector } of a saddle
connection is the vector in the plane obtained by integrating the pullback
of  the planar form $(dx, dy)$, along the saddle connection. We denote the collection of
all holonomy vectors for $\su$ by $V(\su)$. The large scale geometry
of $V(\su)$ has been intensively studied, and one of the main results of
\cite{EsMa} is that there is $c>0$ such that for a.e.\ $\su$ (with
respect to the flat measure), the number $N(T, \su) =|V(\su) \cap B(0,T)|$ satisfies 
\eq{eq: quadratic growth}{
N(T, \su) = cT^2 + o\left(T^2 \right).
}
When this holds we will say that $\su$ {\em satisfies quadratic growth}. 

The main purpose of this paper is to estimate the error term in the above
result, that is to establish that 
$$
N(T, \su) = cT^2 + O\left(T^{2(1-\kappa)}\right)
$$
for some $\kappa>0$. In order to state our result in its full
generality we need to introduce more precise terminology. 

 In this paper, the notations $g = O(f)$ and $g =
O_A(f)$ mean respectively that $f, g$ are functions of a
variable $x$, $A$ is a parameter, and there
is a constant $C$ (depending on $A$) such that for all $x$,
$g(x) \leq Cf(x)$. We will use $f \ll g $ and $f
\ll_A g$ synonymously with $f = O(g)$ and $f = O_A(g)$. 
Let $\cH$ be a stratum of translation surfaces, let $G = \SL_2(\bR)$
and let $\cL \subset \cH$ be the closure of a $G$-orbit in $\cH$. By
recent breakthrough results of Eskin, Mirzakhani and Mohammadi
\cite{EsMi, EsMiMo}, $\cL$ is the intersection of $\cH$ with a linear
suborbifold, and is the support of a smooth ergodic 
probability measure $\mu$, which we will call the {\em flat measure of
  $\cL$.} We will refer to $(\cL, \mu)$ as a {\em locus} (the
terminology `affine invariant manifold' is also in common use). 

A {\em cylinder} on a translation surface is an isometrically embedded
image of the annulus $[a_1,a_2] \times \bR/c\bZ$, for some $a_1<a_2$ and
$c>0$. The image of a curve $\{b\} \times \bR/ c\bZ$ for $a_1<b< a_2$ is
called a {\em waist curve} of the cylinder and the integral along a
waist curve of the
pullback of $(dx,dy)$ is called the {\em holonomy
  vector} of the cylinder. One can also study the asymptotic growth of $V^{\mathrm{cyl}}(\su) \cap
B(0,T)$, where $V^{\mathrm{cyl}}(\su)$ is the collection of holonomy
vectors of cylinders on $\su$. Furthermore, in \cite[\S 3]{EMZ},
Eskin, Masur and Zorich defined 
{\em configurations} which are  a common generalization of saddle
connections and cylinders. We will not need to repeat the definition
of a configuration in this paper; in order to give the idea, we note
three other  examples of configurations: 
(i) $\cC$ consists of a  saddle connection joining some fixed
singularity to itself, (ii) $\cC$ is a saddle connection 
joining distinct fixed singularities, (iii) $\cC$ consists of two homologous saddle
connections joining two distinct fixed singularities, and forming a
slit which disconnects the surface into components with a fixed
topology. For each  configuration $\cC$ one
can then define a collection of holonomy vectors $V^{\cC}(\su)$
of the saddle connections or cylinders comprising the configuration,
and study the asymptotic growth of 
$N^{\cC} (T, \su) = \left|V^{\cC}(\su) \cap B(0,T) \right|$. A remarkable
feature of \cite{EsMa, EMZ} is 
the authors' foresight: they proved their results in an abstract
framework which  later  (in \cite{EsMi}) was proved to be
sufficient to cover all $G$-invariant ergodic measures and all
configurations. Namely, they proved that for any 
locus  $(\cL, \mu)$ and any configuration $\cC$ there is
$c = c(\cL, \cC)$ such that for $\mu$-a.e. $\su \in \cL$ one has
$N^{\cC} (T, \su) = cT^2 + o(T^2).$ 
Furthermore, \cite{EMZ} also discussed counting with multiplicities
(that is, vectors in $\bR^2$ are counted according to the number of
saddle connections which have them as holonomy vectors). In the
notation of this paper $N^{\cC}(T, \su)$ may  refer to 
counting either with or without multiplicity, i.e. the count in
question is assumed to be a part of the data associated with $\cC$. 
Finally, for the case $\cL = \cH$, an algorithm for computing the constants $c$ in the above
asymptotic was described, in terms of so-called Siegel-Veech constants
introduced by Veech in \cite{veech}. 

An additional improvement, due to Vorobets \cite[Thm. 1.9]{vorobets2005periodic},
concerns counting in sectors. Let $\varphi_1 < \varphi_2$ with
$\varphi_2 - \varphi_1 \leq 2\pi$ and let 
$N(T,\su,\varphi_1,\varphi_2)$ denote the cardinality of the
intersection of $V(\su)$ with
the sector 
$$S_{T, \varphi_1, \varphi_2} = \set{r(\cos\varphi, \sin \varphi)\,:\, 0\le
  r \le T, \,\, \varphi_1 \le \varphi \le \varphi_2}\subset
\bR^2.$$ 
Vorobets showed that there is $c>0$ such that for a.e. $\su \in \cH$ (with
respect to the flat measure on $\cH$), $N(T, \su, \varphi_1,
\varphi_2) = c (\varphi_2 - \varphi_1)T^2 + o(T^2)$. Our
main result is an effective version of the above-mentioned
results. Setting $N^{\cC}(T,\su,\varphi_1,\varphi_2)$ for the number
of holonomy vectors corresponding to the  configuration
$\cC$ on $\su$ with holonomy vector in $S_{T, \varphi_1, \varphi_2}$, 
we have:

\begin{theorem}
\label{thm:theorem}
For any locus $(\cL, \mu)$ there is a constant $\kappa>0$ such
that for any 
configuration $\cC$ there is a constant $c>0$ such that for any 
$\varphi_1<\varphi_2$ with $\varphi_2 - \varphi_1 \leq 2\pi$, for
$\mu$-a.e. $\su$ we have 
\eq{eq: what we prove main}{
N^{\cC}(T,\su,\varphi_1,\varphi_2)=\frac{c}{2} (\varphi_2-\varphi_1)
T^2+O_{\su,\varphi_2-\varphi_1}\left(T^{2(1-\kappa)}\right). 
}
\end{theorem}

Here, in the basic case that $\cL = \cH$ is a stratum and $\cC$ is one
saddle connection (i.e. $V^\cC(\su) = V(\su)$), the constant $c$ is
the Siegel-Veech constant of 
\cite{veech} (this is the reason for the denominator 2 appearing in \equ{eq: what we
  prove main}). As we shall see below, $\kappa$ can be
estimated explicitly in terms of the size of the spectral gap in the
unitary representation of $G$ in $L^2(\cL)$.

We have chosen to normalize
our power savings exponent $\kappa$ so that the error is written in
the form $2(1-\kappa)$ rather than $2-\kappa$, that is to estimate the
error as a power of the area growth, in order to permit easier
comparisons with other bounds appearing in the literature on related
problems. Note that in \equ{eq: what we prove main}, the dependence of
the implicit constant in 
the $O$-notation on
$\su$ is unavoidable given the existence of surfaces with
different quadratic growth coefficients.

For a recent application of Theorem \ref{thm:theorem}, see 
\cite{CR}.

The proof of Theorem \ref{thm:theorem} does not give any insight into
the set of full measure of $\su$ which satisfy \equ{eq: what we prove
  main}. In fact it is expected that {\em every} translation surface
$\su$ satisfies quadratic growth (see
\cite{EsMiMo} for a remarkable result in this
direction). 
%{\bf  insert a sentence about a dense $G_\delta$ where the effective statement cannot hold, for any explicit error rate.} 
  Thus it is of
interest to exhibit explicit surfaces which satisfy quadratic
growth with an effective error estimate (in particular, where $\kappa$
is known). It is also of interest to count points in the intersection
of $V(\su)$ with more general subsets of $\R^2$. These questions are
 discussed in \cite{burrin2019effective}.  

\ignore{
A result of this type was proved by Veech in the celebrated
paper \cite{Vee89}. 
The subgroup of $G$ fixing a surface $\su$ is called its {\em Veech
  group}, and will be denoted by $\Gamma_\su$. We will say that $\su$
is a {\em lattice surface} if $\Gamma_\su$ is a lattice in $G$. Veech
showed that lattice surfaces satisfy quadratic growth, and our next
result is an effective version:

\begin{theorem}
\label{thm:theorem2}
Let $\su$ be a lattice surface and let $\cC$ be a
configuration.  Then there are positive 
constants $c$ depending on the orbit $G\su$ and on $\cC$, and 
$\kappa$ depending only on the conjugacy class of $\Gamma_\su$ in $G$,
such that for any $\varphi_1<\varphi_2$ with $\varphi_2 - \varphi_1
\leq 2\pi$ and any  $\su' \in G\su$,
\[
N(T,\su',\varphi_1,\varphi_2)=\frac{\varphi_2-\varphi_1}{2}
c T^2 
+O\left(T^{2-\kappa}\right).
\]
\end{theorem}
Once more $\kappa = \kappa(\Gamma_\su)$ can be estimated 
explicitly in terms of the spectral gap in  
$L^2(G/\Gamma_\su)$. 

\bwpar{What follows are three more results which I claimed in emails, I have not
  written down the proofs yet.} 

A surface $\su$ is called a {\em torus cover } if there is a
continuous map $p: \su
\to \mathbf{t}$, where $\mathbf{t}$ is a translation surface with the
underlying topology of a torus, and a finite set $B
\subset \mathbf{t}$ of {\em branch points}, such that the restriction
of $p$ to $\su \sm p^{-1}(B)$ is a covering map which is a translation
in charts. A torus cover with a unique branch point is a lattice
surface, and hence satisfies quadratic growth. It was shown in
\cite{EMS, EMMdoubleprime} that torus covers (with any number of branch points) satisfy
quadratic growth. Based on an unpublished result of Str\"ombergsson,
we show: 

\begin{theorem}\label{thm:theorem3}
Suppose $\su$ is a torus cover with two branch points. Then the
conclusion of Theorem \ref{thm:theorem2} is satsified for $\su$. 
\end{theorem}
}

The expectation that any translation surface
satisfies quadratic growth, and  that
the constant $c$ appearing  in \equ{eq: quadratic growth} depends only
on the orbit closure $\overline{G \su}$, leads to the expectation that the set of
surfaces satisfying \equ{eq: quadratic growth} is $G$-invariant. Since
the assignment $\su \mapsto V(\su)$ satisfies $V(g\su) = gV(\su)$,
this can be equivalently stated as a problem on counting in ellipses:
in the definition of $N(T, \su)$, one should be 
able to replace Euclidean balls of radius $T$, with dilates of any
fixed ellipse centered at the origin, and the same should be true for $N(T, \su,
\varphi_1, \varphi_2)$. The issue of existence of a full measure
$G$-invariant set of surfaces with quadratic growth  was not discussed in 
\cite{EsMa}, but could probably be derived from the arguments in
\cite{EsMa, vorobets2005periodic}. Moreover, in the case that $\cL$ is a stratum,
it can be derived from a recent result of Athreya, 
Cheung and Masur \cite{ACM}, in combination with an argument of Veech
\cite[Thm. 14.11]{veech}. Using our technique we obtain the
following effective strengthening. 

\begin{theorem}\label{thm: strengthening}
For any locus $(\cL, \mu)$, there is $\kappa>0$ such that for every
configuration $\cC$ there is  $c>0$ such that for 
$\mu$-a.e. $\su$, for every $\varphi_1< \varphi_2$ with
$\varphi_2-\varphi_1 \leq 2\pi$,  and for every $g
\in G$, 
$$
N^{\cC}(T, g\su, \varphi_1, \varphi_2) = \frac{c}{2}
(\varphi_2-\varphi_1)T^2 + O_{\su,\varphi_2-\varphi_1,g}\left(T^{2(1-\kappa)} \right).
$$
\end{theorem}

Note that Theorem \ref{thm: strengthening} implies Theorem
\ref{thm:theorem}, but we present the proof of Theorem
\ref{thm:theorem} separately. This is because the proof of Theorem
\ref{thm: strengthening} presents additional technicalities which may
obscure the main ideas, and also because our proof of Theorem \ref{thm:
  strengthening} gives slightly weaker estimates on $\kappa$.

\subsection{Ingredients of the proofs}
Our proof of Theorem \ref{thm:theorem} follows the strategy of
\cite{EsMa} (which in turn was inspired by \cite{EsMaMo, veech}) of
reducing the counting problem to an ergodic theoretic problem
regarding the convergence of the translated circle averages $\pi_\cL(\Sigma_t) f
(\su) = \int_K f(a_t k\su) d m_K$ (the notation is introduced in \S
\ref{subsec: circle averages}), as $t \to 
\infty$. In the treatment of \cite{EsMa}, $f$ is the Siegel-Veech
transform of an indicator of a rectangle in $\bR^2$, and the required 
convergence of $\pi_\cL(\Sigma_t) f(\su)$ was proved by replacing $f$
with a smoothed version of $f$, developing various estimates to bound
the amount of time the translated circle average spends outside large compact
subsets of $\cH$, and appealing to a pointwise ergodic theorem of the
first-named author (see \cite{nevo2017equidistribution}). 

Our proof of Theorem \ref{thm:theorem} uses all of the above
ingredients  and more. The essential new ingredient is the
fact that any $(\cL, \mu)$
possesses a spectral gap (see \S \ref{sec: four} for the
definition). This was proved by Avila, Gou\"ezel and Yoccoz
\cite{AGY} for the case 
of strata, and by Avila and  Gou\"ezel \cite{AG} for general loci (again, in an
abstract framework, as \cite{AG} also preceded \cite{EsMi}). Using
the spectral gap it is possible to obtain an effective estimate 
of the difference  $|\pi_\cL(\Sigma_t) f (\su) - \int_{\cL} f d\mu|$, in case $f$ is a
$K$-smooth function and $t$ is large enough (depending on $\su$ and
$f$). See \S \ref{sec: four}
for the definition of $K$-smooth functions. The estimate is valid for 
$\su$ in a set of large measure depending on
$f$ and $t$. Using a Borel-Cantelli argument (see Theorem
\ref{thm:ergodic}) we upgrade this to a set of full measure and a
countable collection of $K$-smooth functions, which we then use in
order to estimate effectively the integrals appearing in the counting
problem, and thus the numbers $N^{\cC}(T_n, \su, \varphi_1,
\varphi_2)$ for a countable collections of radii $(T_n)$.  In order to
pass from a countable collection of functions to 
the results, it is advantageous to replace the rectangle used in
\cite{EsMa}, or the trapezoid used in \cite{eskin2006counting}, with a
triangle with an apex at the origin. 

Theorem \ref{thm: strengthening} improves Theorem \ref{thm:theorem} in
two ways: uniform counting with an error term in all sectors and in all ellipses. These
improvements require two additional ingredients. First we note that  
the same Borel-Cantelli
argument, and further approximation arguments, make it possible
to use countably many functions in order to approximate all sectors
and all ellipses simultaneously. That is, instead of working only with
a countable set of radii, we work with a countable set of radii, a
countable set of ellipses, and a countable collections of
sectors. Furthermore, for
uniform counting in ellipses, we replace the circle 
averages with ellipse averages $\pi_\cL\left(\Sigma^{(g)}_t \right) f(\su)= \int_K
f(a_t k g\su) dm_K$, and obtain
an estimate on the rate at which $\pi_\cL \left(\Sigma^{(g)}_t
\right)(\su) \to \int_{\cL} f d\mu $, which is uniform as $g$ ranges
over compact subsets of $G$.

\ignore{
\subsection{History of the problem}
Eskin and Masur \cite{EsMa} studied the asymptotics of the number
of saddle connections of almost every flat surface. They established
that  
$N(T,\su)$ is almost surely asymptotic to a multiple of $T^2$. Earlier, Masur has established 
two sided bounds for $N(T,\su)$ which are proportional to $T^2$. Veech
\cite{veech} has established convergence of $N(T,\su)/T^2$ in the
$L^1(\cL,\mu)$-norm to the constant $\pi c(\mu)$. Everywhere
convergence in the  case of lattice surfaces was established  in
\cite{Vee89}, in non-effective form.   

 The almost sure result established in \cite{EsMa} is based on a
 pointwise ergodic theorem proved for this purpose, see \cite{N}. Our
 approach in the present paper is based on a much sharper pointwise
 ergodic theorem, which provides an effective and explicit rate of
 convergence. This utilizes the fundamental result of Avila
 Gou{\"e}zel and Yoccoz \cite{AGY} that the action of $G$ on
 $(\cL,\mu)$ has a spectral gap 

{\bf More history ............}
\bwpar{Mention: Veech and Gutkin-Judge for the Veech surface
  result. Avila Gou\"ezel for a generalization of \cite{AGY}. Explicit
cases: Eskin-Masur-Schmoll, Eskin-Marklof-Morris,
Bainbridge. Eskin-Mirzakhani-Mohammadi for all surfaces with an extra
Cesaro. Eskin-Masur-Zorich for Siegel Veech constants and for general
configurations. Vorobets for 
sectors. } 

}

\subsection{Acknowledgements} 
The authors were supported by ERC starter grant DLGAPS 279893,ISF
grant 2095/15 and SNF grant P2EZP2 168823.

\section{Preliminaries} \label{sec: preliminaries}
In this section we will collect results which we will need concerning
the moduli space of translation surfaces. 

\subsection{The Siegel-Veech formula and the function $\ell$} 
We recall the Siegel-Veech summation formula: 
\begin{theorem}\cite[Thm. 0.5]{veech}
\label{thm:siegelveech}
For any locus $\cL$ with flat measure $\mu$, and any
configuration $\cC$, 
there exists $c =c(\cL, \cC)>0$ (called a Siegel-Veech constant) such
that for any $\psi\geq0$ Borel 
measurable on $\bR^2$, if we let $\widehat{\psi}(\su)=\sum_{v\in
  V^{\cC} (\su)}\psi(v)$, then
\[
\int_\cL \widehat{\psi}(\su)d\mu(\su)= c \int_{\bR^2} \psi(x) dx.
\]
\end{theorem}
We stress that the definition of $\widehat{\psi}$ depends on a choice
of configuration $\cC$, but this choice will not play an important
role in what
follows, and will be suppressed from the notation. 

Let $\ell(\su)$ be the Euclidean length of a shortest saddle
connection in $\su$.  
Building on earlier work in \cite{M90} and \cite{EsMaMo}, a 
fundamental bound on the number of
saddle connections in a compact set  was established by Eskin and Masur
as follows.  
\begin{theorem}\cite[Theorem 5.1]{EsMa}
\label{thm:boundbyell}
For any stratum $\cH$, any configuration $\cC$, any compact set
$B\subset \bR^2$, any $\su \in \cH$, and any
$\consta\label{exp:boundbyell}>1$, 
\[
\left|V^{\cC}(\su)\cap B \right|\ll_{ \cH, B, 
 \ref{exp:boundbyell}} \ell(\su)^{-\ref{exp:boundbyell}}.
\]
%Moreover, if $B\subset B_R(0)$ for $R\leq 1$ then
%\[
%\left|V^{\cC}(\su)\cap B \right|\ll_{ \cH, \cC, 
% \ref{exp:boundbyell}} R^{\ref{exp:boundbyell}}\ell(\su)^{-\ref{exp:boundbyell}}.
%\]
\end{theorem}
Note that in \cite{EsMa}, the bound was only stated for the set 
$V(\su)$ of all saddle connection holonomies, that is the case in
which the configuration $\cC$ consists 
of any saddle connection; however since any 
cylinder contains saddle connections along its boundary, the bound
for $V(\su)$ implies the same bound for $V^{\cC}(\su)$ for any
configuration $\cC$.

\subsection{Translated circle averages} \label{subsec: circle averages}
Consider the elements  
\eq{eq: consider}{a_t = \begin{pmatrix} e^t & 0 \\ 0 &
    e^{-t} \end{pmatrix}, \ 
\ k_\theta = \begin{pmatrix} \cos \theta &  -\sin \theta \\ \sin
  \theta & \cos \theta \end{pmatrix}
}
 and
let $K=\{k_\theta : \theta\in [0, 2\pi)\} \subset G$. When $G$ acts
ergodically by measure preserving transformations on a standard Borel
probability space $(X,\mu)$, we will say that $(X, \mu)$ is an ergodic
p.m.p.\ $G$-space. We let $\pi_X$ denote the
unitary representation of $G$ in 
$L^2(X)$, given by $\pi_X(g)f(x)=f(g^{-1}x)$. 
 We extend $\pi_X$ to a
representation of the convolution algebra $M(G)$ of bounded complex
Borel measures on $G$. Each $\sigma\in M(G)$ acts as an operator on
$L^2(X)$ via the formula
$$\pi_X(\sigma)f(x)=\int_G f(g^{-1} x) d\sigma(g), \ \text{ for } f\in
L^2(X).$$ 
For any two measures  
$\sigma_1,\sigma_2\in M(G)$, we have $\pi_X(\sigma_1\ast
\sigma_2)=\pi_X(\sigma_1)\circ\pi_X(\sigma_2)$.  

Let $m_K$ denote the probability Haar measure on the circle $K$ given
in coordinates by  $\frac{1}{2\pi} d\theta,$ and denote  
the probability measure $m_K\ast \delta_{a_{-t}}$ by $\Sigma_t$. Thus
for $f:\cL\to\bR$  
\[
\pi_X(\Sigma_t)f(\su)= \int_K f(a_tk\su) dm_K(k).
\]
An important property of integrability of the function $\ell$, and a
bound on its translated circle  averages, were established by Eskin
and Masur:
\begin{theorem}[See \cite{EsMa} Thm. 5.2, Lem. 5.5 and
  \cite{veech}, Cor 2.8]
\label{thm:integrability}
For any $\su\in \cL$, and for any $1\leq\consta\label{exp:integrability}<2$,
\begin{equation}\label{eq: integrability}
\sup_{t>0} \pi_\cL(\Sigma_t)\left(\ell(\su)^{-\ref{exp:integrability}}\right)<\infty\,.
\end{equation}
The bound can be taken to be uniform as $\su$ ranges over compact sets in
$\cL$. Furthermore, for any locus $(\cL, \mu)$, we have
$\ell(\cdot)^{-\ref{exp:integrability}}\in L^1(\cL,\mu)$. 
\end{theorem}

To account for sectors, we will use the family of measures on the circle 
\[
\pi_{X}(\Sigma_{\nu,t})f(\su)= \int_K f(a_tk\su) \nu(k)dm_K(k)
\]
where $\nu$ is a bounded density on $K$. In fact, in this paper
$\nu$ will be a 
characteristic function of an angular sector
$I=I_{\phi_1,\phi_2}=[\varphi_1,\varphi_2]$, so that $d\nu=\chi_I dm_K$.  We will also consider below densities $\nu_t$ corresponding to intervals which constitute a slight contraction or a 
slight expansion of $I$. It is clear that \eqref{eq: integrability} 
also holds for such $\Sigma_{\nu,t}$, uniformly for all
$\nu\leq 1$. 

\section{Spectral gap and pointwise ergodic theorem}\label{sec: four}
\subsection{Spectral gap and matrix coefficients estimate}
Let $(X, \mu)$ be an ergodic p.m.p.\ $G$-space,
and denote by $L^2_0(X, \mu)$ the zero mean functions in 
$L^2(X, \mu)$. By ergodicity, there are no nonzero invariant vectors
in $L^2_0(X, \mu)$. The action is said to have a {\em spectral gap} if the
associated unitary representation of $G$ is isolated
from the trivial representation; equivalently, there does not exist a
sequence of unit vectors $(u_j)_{j \in \N}$  in $L^2_0(X,\mu)$ which is
{\em asymptotically invariant} under the representation, namely such that
$\lim_{j\to 0}\norm{\pi_X(g)u_j-u_j}=0$ for every $g$ in $G$.
Note that if $(u_j)_{j \in \N}$ is an asymptotically invariant sequence, and $K$ 
is a compact subgroup of $G$, then $\displaystyle{v_j
=\frac{\pi_X(m_K)u_j}{\norm{\pi_X(m_K)u_j}}}$ is well-defined
for all but finitely many indices, and 
$(v_j)_{j \in \N}$ is an asymptotically invariant sequence consisting
of $\pi_X(K)$-invariant unit vectors.

Our results are based on the following important
result:
\begin{theorem}\cite{AGY, AG}
The representation of $G$ on $L^2_0(\cL,\mu)$ possesses a spectral gap. 
\end{theorem}

The functions  $g \mapsto \langle \pi_X(g)f_1,f_2 \rangle$, for $f_i
\in L^2(X, \mu)$, are known as {\em matrix coefficients} for the
action on $(X, \mu)$. $f$ is called a $K$-eigenvector if there exists a character
$\chi$ of $K$ such that $\pi_X(k) f= \chi(k)f$ for all $k\in K$. 
 If $f$ is a finite linear combination of
$K$-eigenvectors, it is called {\em $K$-finite}. 
Fix $\omega= \left(\begin{smallmatrix} 0 & -1 \\ 1 & 0  \end{smallmatrix}\right)$ as a
generator  of the Lie algebra of $K$. 
A function $f\in L^2(X)$ is called {\em $K$-smooth of degree one} if
\eq{eq: pi notation}{
\pi_X(\omega) f\stackrel{\on{Def}}{=}\lim_{\phi\to
  0}\frac{1}{\phi}\left(\pi_X(\exp(\phi\omega)f-f\right)
}
exists, where the 
convergence is with respect to the $L^2(X)$-norm (one may also consider the
obvious extension to smoothness of degree $d$ for $\omega$, but
we will not need this). 
Define the
(degree one) Sobolev norm by 
\[
\cS_K(f)^2=\|f\|_2^2+\|\pi_X(\omega) f\|_2^2.
\]
We denote the space of $K$-Sobolev functions with finite
$\cS_K(f)$-norm by $\cS_K(X)$, and set 
$$\cS_{K,0}(X)  = \cS_K(X, \mu) \cap L^2_0(X, \mu).$$ 

In the special case $G=\SL_2(\bR)$ the spectral gap condition implies
the following explicit quantitative estimate.

\begin{theorem}\label{CHH-estimate}
Let $G = \SL_2(\R)$ and let $(X, \mu)$ be an ergodic p.m.p. $G$-space with a
spectral gap. Then there are positive $C, \lambda$ such that for any $f_1,f_2\in
L^2_0(X)$ which are $K$-eigenvectors, and for any $g \in G,$ 
written in Cartan polar coordinates as $g=k_1a_t k_2$, we have
\eq{eq: in terms of KAK}
{
\abs{\langle \pi_X(g) f_1,f_2\rangle} \le C
e^{-\abs{t}\lambda}\|f_1\|_2\|f_2\|_2\, 
}
for any $|t|\geq1$.
Furthermore, for any matrix norm on $\on{Mat}_2(\R)$, and $\lambda$
for which \equ{eq: in terms of KAK} holds, and any $K$-Sobolev functions $f_1,f_2 \in \cS_{K,0}(X)$, 
\begin{equation}\label{eq:norm2}
\abs{\langle \pi_X(g) f_1,f_2\rangle} \ll_\lambda
 \|g\|^{-\lambda}\cS_K(f_1)\cS_K(f_2).
\end{equation}

\end{theorem}

The supremum of $\lambda > 0$ for which one can find $C$ such that 
\equ{eq: in terms of KAK}  is satisfied for $K$-eigenvectors $f_1,
f_2$, will be denoted by $\lambda_X$ 
and will be called the {\em size of the spectral gap}. 
Note that the results of \cite{AGY, AG} do not give explicit bounds
on the size of the spectral gap. 

Theorem \ref{CHH-estimate} is well-known to experts as part of the
general theory of 
unitary representations of simple Lie groups, but a convenient
reference for the case at hand is hard to come by. We give a proof
below. 
The proof we give below establishes \equ{eq: in terms of KAK}
for $K$-eigenvectors and \eqref{eq:norm2}
for $K$-Sobolev functions
in any unitary
representation of $G$ with a spectral gap, not only the
representations arising from p.m.p. actions on probability spaces.  
\begin{proof}
We use the concise exposition of the unitary representation theory of
$\SLR[2]$ in \cite{howetan}, where   
a full parameterization of the unitary dual $\widehat{G}$ is given in
\cite[Ch. III, \S 1.3, Thm. 1.3.1]{howetan}, and an explicit
construction of the corresponding irreducible unitary representations
is given in \cite[Ch. V, \S 3.1]{howetan}. The unitary dual can be
divided to four parts: the principal series (spherical and
non-spherical), the complementary series, the countable set of
discrete series representations and the two representations referred
to as `limits of discrete series'.   
Let $g=k_1a_t k_2$ and consider the matrix coefficient $\langle
\tau(g)v_1,v_2 \rangle$, with $\tau$ an irreducible non-trivial
unitary representation and $v_1$, $v_2$ being $K$-eigenvectors of unit
norm, including the case where $v_1$, $v_2$ are $K$-invariant. We shall always assume that $|t|\geq 1$.

For $\tau$ in the principal series the matrix coefficients are bounded by $C \abs{t}\exp (-\abs{t})$, by \cite[Ch. V. \S 3.1, eqs. (3.1.2),(3.1.4)]{howetan}, noting that the $K$-eigenvectors are just the characters of the circle group and hence are uniformly bounded functions, and that $C$ is uniform in this case. For $\tau$ a discrete series representation, the same uniform bound holds by \cite[Ch. V, \S 3.2, Thm. 3.2.1]{howetan}. This follows since a discrete series representation is a subrepresentation of the regular representation, using also the bound of the Harish-Chandra $\Xi$-function provided in the first estimate of \cite[Ch. V, \S 3.1. Prop. 3.1.5]{howetan}.

The complementary series representations are parameterized by
$\tau_s$, $0 < s < 1$, and the matrix coefficients of $K$-eigenvectors
of unit norms are bounded by $C_{\tau_s}  \abs{t}\exp(-(1-s)\abs{t})$,
using the second estimate in \cite[Ch. V, \S
3.1. Prop. 3.1.5]{howetan} (with $C_{\tau_s}$ possibly depending on $s$
according to this estimate). In particular, for each such $s$ there
exists an integer $n(s)$ such that the matrix coefficients, raised to
the $n(s)$-power, are in $L^2(G)$. It follows that the tensor power
representation $\tau_s^{n(s)}$ embeds as a subrepresentation of the
regular representation by \cite[Ch. V, \S 1.2,
Cor. 1.2.4]{howetan}. But then $(\langle \tau(g)v_1,v_2
\rangle)^{n(s)}$ satisfies the bound that a matrix coefficient
associated with two $K$-eigenvectors in the regular representation
satisfies, which is given in \cite[Ch. V, \S 3.2,
Thm. 3.2.1]{howetan}. It follows that
\eq{eq: 3.4}{\abs{\langle \tau_s(g)v_1,v_2
    \rangle}\le  Cte^{-\frac{\abs{t}}{n(s)}},}
with $C$ uniform over $0 < s <
1$.  

Finally, as to the two `limits of discrete series' representations, by
\cite[Ch. V, Exer. 8, p. 242]{howetan} the relevant matrix
coefficients are all in $L^{2+\vre}$ for $\vre > 0$ and so
also in  $L^{3}$, and hence the argument of the previous paragraph
applies to give \equ{eq: 3.4}, with $3$ replacing $n(s)$, and $C$ uniform.

We shall now use arguments appearing in \cite{ratner1987rate}. 
An arbitrary (separable strongly continuous) unitary representation $\pi$ of $G$ 
can decomposed as a direct integral of non-trivial irreducible representations (see
\cite [p.\ 272ff]{ratner1987rate}): Let $(Y,\zeta)$ be a standard Borel space 
and suppose there are non-trivial irreducible representations $\tau_y$
of $G$ for all $y\in Y$ defined on separable Hilbert space $H^y$ and
some choice of orthonormal basis $\{\phi^y_n\}_{n\in\Z}$ for each
$H^y$. 
We call a function (or more precisely, a section) $f$ on $Y$ with
$f^y\in H^y$ measurable w.r.t\ this choice of bases,  
if the inner products $\langle f^y,\phi^y_n\rangle$ are measurable for any $n$. 
The collection $H$ of functions for which $\int_Y \|f^y\|_y^2
d\zeta(y)$ is finite constitutes a separable Hilbert space,  
with inner product $\langle f,h\rangle = \int_Y \langle f^y,h^y\rangle d\zeta(y)$ 
and the $G$-representation $\int_{ Y}\tau_y d\zeta(y)=\pi$
defined by $\left(\pi(g)f\right)^y=\tau_y(g)f^y$, for
$\zeta$-a.e. $y\in Y$.   
Conversely, for any representation $\pi$ there exists such $(Y,\zeta)$ 
so that $\pi$ is unitarily equivalent to $\int_{Y}\tau_y d\zeta(y)$,
and so we assume that the representation $\pi_X$ is disintegrated in
such a manner. 
We may further decompose $f^y=\sum f^y_n$ into isotypic components
with respect to $K$ by assuming the basis of $H^y$ consists of
$K$-eigenvectors (see \cite [Lem.  1.1]{ratner1987rate}).  If $f$ and
$h$ are $K$-eigenvectors of $ 
\pi_X$, then their components $f^y$ and $h^y$ in the representations
$\tau_y$ are $K$-eigenvectors of $\tau_y$ affording the corresponding
characters (for $\zeta$-almost every $y\in Y$).  

Let us now note that a sequence $u_j$ of unit vectors is
asymptotically invariant if and only if $\left| \langle \pi_X(g)
  u_j,u_j\rangle \right| \to 1$ for every $g\in G$ (or equivalently
uniformly over compact subsets of $G$). If $u_j$ are $K$-invariant, so
are their direct components $u_j^y$, for $\zeta$-a.e. $y\in
\widehat{G}$, and then  
$\abs{\langle \pi_X(g) u_j,u_j\rangle } $ obeys the bounds \equ{eq:
  3.4}. 

Suppose now that the spectral measure $\zeta$ assigns zero measure to
the set of complementary series representations given by
$\set{\tau_s\,;\, s>s_X}$, for some $0 < s_X < 1$. The bounds \equ{eq:
  3.4} then immediately imply that there exist positive $C$ and
$\lambda$ such for every $K$-eigenvectors $f,h\in L^2_0(X)$  

\begin{equation}\label{eq:disint}
	  \begin{split}
\left| \langle \pi_X(g) f,h\rangle \right| & = \left|  \int_{Y}
  \langle\tau_y(g)f^y,h^y \rangle d\zeta(y) \right|
\\ & \leq Ce^{-\lambda\abs{t}}
\int_{Y}\|f^y\|_y\|h^y\|_y d\zeta(y) \\ &  \le Ce^{-\lambda\abs{t}}
\|f\|_2\|h\|_2,
\end{split}
\end{equation}
using the Cauchy-Schwarz inequality for the last inequality. In particular, it follows that the representation does not admit an asymptotically invariant sequence of $K$-invariant unit vectors in that case. 

Conversely, if $\zeta\left(\set{\tau_s\,;\, s > 1-\frac{1}{j}}\right) > 0$  for every $j\ge 1$ ($\tau_s$, $0 < s < 1$ being the complementary series), then $\pi_X$ does admit an asymptotically invariant  sequence $u_j$ of $K$-invariant unit vectors.  Indeed, each $\tau_s$ admits a unique $K$-invariant unit vector $v^s$, up to a multiplication by a complex number of absolute value $1$, and independently of the choice of this scalar we have $ \langle \tau_s(g) v^s,v^s\rangle =\Phi_s(g)$. Here $\Phi_s$ is the standard positive-definite and positive spherical function associated with the spherical representation $\tau_s$. It is well known that for any fixed $g=k_1 a_t k_2$ we have $\lim_{s\to 1} \Phi_s(g)= 1$.  
This follows immediately from the integral representation for the positive spherical function  $\Phi_s$, normalized so that $\Phi_s(e)=1$. Indeed, in the present case, for $f=h=1$ (corresponding to the trivial character of $K$), the inequality \cite[eq. 3.1.2, p. 215]{howetan} is in fact an identity because of the positivity of the integrand, and so  
$$\Phi_s(a_{\abs{t}})=\frac{1}{2\pi}\int_0^{2\pi}\left(e^{-2\abs{t}}\cos^2 \phi+e^{2\abs{t}}\sin^2 \phi\right)^{-\frac12(1-s)}d\phi\,.
$$
%
%$$=\frac{2e^{-(1-s)\abs{t}}}{\pi}\int_0^{\pi/2}\abs{e^{-2\abs{t}}\cos^2 \phi+e^{2\abs{t}}\sin^2\phi}^{-\frac12 (1-s)}d\phi
%$$
%$$\frac{e^{-(1-s)\abs{t}}}{2\pi}\int_0^{2\pi}\abs{e^{2\abs{t}}-2\sinh (2\abs{t}) \cos^2\phi	}^{-\frac12 (1-s)}d\phi
%$$

Now let $u_j$ be any $K$-invariant unit vector $u_j$ in the
subrepresentation of $\pi_X$ given by  $\pi_j=\int_{\set{ s >
    1-\frac{1}{j}}}\tau_s d\zeta(\tau_s)$. Such vectors do exist since
the direct integral of  irreducible representations each containing a
$K$-invariant unit vector has the same property, and the sequence
$u_j$ satisfies, for each fixed $g$:   
$$\inn{\pi_j(g)u_j,u_j}=\int_{\set{ s > 1-\frac{1}{j}}}\inn{\tau_s(g)v^s,v^s} \norm{u^s_j}^2d\zeta(\tau_s)$$
$$=
\int_{\set{ s > 1-\frac{1}{j}}}\Phi_s(g) \norm{u^s_j}^2d\zeta(\tau_s)\longrightarrow  \int_{\set{ s > 1-\frac{1}{j}}} \norm{u^s_j}^2d\zeta(\tau_s)
=1\text{,  as  } j \to \infty\,.$$

Denoting by $s_X$ the infimum over  $0<s < 1$ for which $\zeta\left(\set{\tau_s'\,;\,  s' >s}\right)=0$, the previous argument shows that $\pi_X$ has a spectral gap if and only if $0\le s_X < 1$.  We define $\lambda_X=1/n(s_X)$, and then we can choose any $0< \lambda < \lambda_X$ and inequality (\ref{eq:disint}) is satisfied. Note that the $\abs{t}$-factor appearing in the bound of matrix coefficients of the complementary series makes the constant $C$ used in \eqref{eq:disint} depend on the choice of $\lambda$. We will therefore write the bound as $\ll_\lambda$ from now on.

Moving on to $K$-smooth functions, we note that $\omega$ defines an operator $\pi_X(\omega)$ acting on $K$-smooth vectors in $L^2(X)$, and the action of this operator is equivariant with respect to the decomposition, namely $\langle\pi_X(\omega)f,h\rangle =\int_{Y}\langle\tau_y(\omega)f^y,h^y\rangle d\zeta(y)$. Equivalently, for $\zeta$-almost every $y$, $f^y$ is $K$-smooth and $(\pi_X(\omega)f)^y=\tau_y(\omega)f^y$ (\cite{ratner1987rate}[Lem. 1.2]).
Now note that if  $v_n$ is a $K$-eigenvector with character $e^{in\theta}$ for a representation $\tau$, then 
$$ \tau(\omega) v_n= \left
. \frac{d}{d\phi}\right|_{\phi=0} \tau(\exp(\phi\omega))v_n=in v_n\,,$$
and hence $\norm{v_n}= \frac{1}{n}\norm{\tau(\omega)v_n}$. Furthermore, if $v=\sum_{n\in \Z}v_n$ is the decomposition of $v$ to isotypic components (whose components are mutually orthogonal), then using the previous identity and the Cauchy-Schwarz inequality 

$$\sum_{n\in \Z}\norm{v_n}= \norm{v_0}+\sum_{n\neq 0}\frac{1}{n}\norm{\tau(\omega)v_n}
\le \norm{v_0}+
\left(\frac{\pi^2}{6}\right)^{1/2}\left(\sum_{n\neq 0}\norm{\tau(\omega)v_n}^2\right)^{1/2}\le 2(\norm{v_0}+\norm{\tau(\omega)v}).$$

Let $f$ and $h$ be two $K$-smooth vectors in $L^2_0(X)$, and decompose their direct integral constituents $f^y$ and $h^y$ into their isotypic components: $f^y=\sum f^y_n$, and $h^y=\sum h^y_n$. We have 

\[
\abs{\langle \tau^y(g) f^y,h^y\rangle}
=\abs{ \sum_{n,m} \langle \tau^y(g) f^y_n,h^y_m\rangle}
\ll_\lambda
 e^{-\lambda\abs{t}}  \sum_{n,m} \|f^y_n\|_y\|h^y_m\|_y
 \]
 \[
\ll_\lambda
 e^{-\lambda\abs{t}} \left( \sum_{n} \|f^y_n\|_y\right)\left( \sum_{m} \|h^y_m\|_y\right)
  \ll_\lambda  e^{-\lambda\abs{t}} (\norm{f_0^y}_y+\norm{\tau^y(\omega)f^y}_y)\cdot(\norm{h_0^y}_y+\norm{\tau^y(\omega)f^y}_y).
\]
Since $\norm{f_0^y}_y\le \norm{f^y}_y$ $\norm{h_0^y}_y\le \norm{h^y}_y$, and $(a+b)^2\leq 2(a^2+b^2)$, integrating  w.r.t. the measure $\zeta$ as in \eqref{eq:disint}, we see that given $f,h\in \mathcal{S}_{K,0}(X)$, their matrix coefficient in absolute value $\left| \langle \pi_X(g) f,h\rangle \right|$ is bounded by
\begin{equation*}\label{eq:disint2}
\begin{split}
&\ll_\lambda e^{-\lambda\abs{t}} 
\left(\int_{Y}(\|f^y_0\|_y+\|\tau_y(\omega)f^y\|_y)^2 d\zeta(y)\right)^{1/2} 
\left(\int_{Y}(\|h^y_0\|_y+\|\tau_y(\omega)h^y\|_y)^2 d\zeta(y)\right)^{1/2} \\
&\ll_\lambda e^{-\lambda\abs{t}} 
\left(\int_{Y} \left( \|f^y\|_y^2+\|\tau_y(\omega)f^y\|^2_y\right) d\zeta(y)\right)^{1/2}
\left(\int_{Y}\left(\|h^y\|_y^2+\|\tau_y(\omega)h^y\|^2_y \right)
  d\zeta(y)\right)^{1/2} \\ 
&= e^{-\lambda\abs{t}} 
\left(\|f\|_2^2+\|\pi_X(\omega)f\|_2^2\right)^{1/2}
\left(\|h\|_2^2+\|\pi_X(\omega)h\|_2^2\right)^{1/2}
=e^{-\lambda\abs{t}} 
S_K (f)S_K(h). 
\end{split}
\end{equation*}

Finally, we may choose to replace the quantity $e^{-\abs{t}}$ with any
matrix norm, and conclude that for any two $K$-eigenfunctions in
$L^2_0(X)$, or any two $K$-smooth functions: 
\begin{equation}\label{eq: in terms of norm}
\abs{\langle \pi_X(g) f,h\rangle} \ll_\lambda
\norm{g}^{-\lambda}\|f\|_2\|h\|_2.
\end{equation}
This follows from the fact that the Euclidean (sum-of-squares) norm on
$\on{Mat}_2(\R)$ satisfies $\|k_1a_t k_2\|_2=\sqrt{2\cosh 2t}\ll
\sqrt{2}e^{\abs{t}}$, and any two linear norms on $\on{Mat}_2(\R)$ are
equivalent.  
\end{proof}

\subsection{Effective pointwise ergodic theorem}
As shown by Eskin, Margulis and Mozes, from estimates such as those in
Theorem \ref{CHH-estimate}, one  can derive an estimate 
for the norm of the operator $\pi_X(\Sigma_{\nu,t})$, $\nu\leq1$, viewed as an
operator from the $K$-Sobolev space $\cS_{K,0}(X)$ to $L^2(X)$.

\begin{theorem}[See \cite{EsMaMo}(3.32) and \cite{veech},  \S14]
\label{thm:L2}
Let $G = \SL_2(\R)$ and let $(X, \mu)$ be a p.m.p.\ $G$-space with a
spectral gap of size $\lambda_X$. Then for any $\lambda< \lambda_X$,
there exists $C_\lambda>0$ such that for any %$\lambda'>0$, any
interval $I \subset \bS^1 \simeq K$ of length $|I|\neq 0$, any
$f\in \cS_{K,0}(X)$, and all $t> \frac12\log \frac{1}{\abs{I}}$, 
\eq{eq: norm estimate}{
\| \pi_{X}(\Sigma_{\nu,t}) f\|_2^2\leq C_\lambda % (|I|e^{-
                                % 4\lambda\lambda't}\cS_K(f)^2+e^{-2(1-\lambda')t}\|f\|_2^2). 
e^{-2\lambda \eta t} \cS_K(f)^{2} \abs{I}^{2-\lambda\eta},
}
where $\eta = \frac{1}{\lambda+1}$ and $\nu$ is the indicator function of
the interval $I$. 
 %satisfies $\abs{I}>e^{-2t}$.
%\max{(\nu\star\tilde{\nu})}
%}
%where 
%\[
%\tilde{\nu}(k) = \nu(k^{-1}) \text{ and } \nu\star\tilde{\nu}(k)=\int_K \nu(kk')\nu(k')dm_K(k').
%\]
\end{theorem}

We
note that if one normalizes $\nu$ to be the density of a probability
measure then the quality of the rate in \equ{eq: norm estimate}
diminishes as the length of the interval 
decreases. We will not normalize $\nu$ in this way because it will
turn out to be less natural for some geometric considerations involved
in the counting problem. 

For completeness, and in order to have precise control of constants,
we repeat the argument found in \cite{EsMaMo, veech}.  

\begin{proof} Let $\nu^{\ast}(g)=\nu(g^{-1})$.
Then
$\nu^{\ast}\ast \nu(k)=\int_K\nu(k')\nu(k'k)dm_K(k')=D_\nu(k)$, by
$G$-invariance of the measure $\mu$ one has  
\eq{eq: coefficient}{
\|\pi_{X}(\Sigma_{\nu,t})f\|_2^2=\langle\pi_{X}(\delta_{a_{t}}\ast\nu^{\ast}\ast \nu
\ast\delta_{a_{-t}})f,f\rangle = \int_X \int_K f(a_t k a_{-t}x)
\bar{f}(x) D_\nu(k)dm_K(k) d\mu(x). 
}
Replacing $I$ if necessary by a disjoint union of at most $8$ subinterval of length bounded by $\pi/4$, 
%with a slightly smaller interval will only require a small increase in $C$, 
without loss of generality we can assume that 
$|I|\le \pi/4$. Then using a rotation we can assume that $I=[0,\phi_I]
\subset [0,\pi/4]$.  
We identify $K$ with $\bS^1$ using \equ{eq: consider} so that
$\nu(K)=\frac{\abs{I}}{2\pi}\leq\abs{I}$.  
For a parameter
$0 < \lambda' < 1$ to be fixed below, we set
\( 
J=\{\phi\in I: \sin\phi \cdot e^{2t}<e^{2\lambda't}\}.
\)
To put ourselves in the case that $J$ is a proper subset of $I$, we assume 
that $\sin \phi_I > e^{2t(\lambda'-1)}$, which implies that 
$1> \frac{\pi}{4}\ge \phi_I=\abs{I}> \sin \phi_I  >  e^{2t(\lambda'-1)}$, and thus $0< \lambda' <  \frac{\log\abs{I}}{2t}+1 < 1$.
% In view of the assumption $t>
%\frac12 \log \frac{1}{\abs{I}}$, this implies $0 < \lambda^\prime < 1$. 

Write $k=k_\phi$ so that, using the supremum norm on $\on{Mat}_2(\R)$, we
have 
\(
\|a_tka_{-t}\| \geq e^{2\lambda' t}
\)
for $k\in I\setminus J.$  Since $I\subset [0,\frac{\pi}{4}]$,     
$D_\nu$ can be computed using convolution on $\bR$, and   
% and is symmetric around the origin. 
since $\phi \le 2\sin \phi$ in the interval $I$, and $\| D_\nu\|_\infty \leq \abs{I}$, we
conclude that $\int_J D_\nu(k) dm_K <2\abs{I} e^{2t(\lambda'-1)}$. Furthermore, clearly 
$\int_{I \setminus J} D_\nu(k) dm_K \leq \abs{I}^2$. By Fubini, the matrix
coefficient \equ{eq: coefficient} is equal to 
$$
\int_J \langle \pi_X(a_tka_{-t})f, f\rangle D_\nu(k) dm_K(k) + \int_{I \sm J}
\langle \pi_X(a_tka_{-t})f, f\rangle D_\nu(k) dm_K(k). 
$$
We apply the previous estimate to the integral over $J$, and apply 
(\ref{eq:norm2}) to the integral over $I \sm J$, to arrive at
\[
\|\pi_{X}(\Sigma_{\nu,t})f\|_2^2\ll
\abs{I} e^{2t(\lambda'-1)}\|f\|^2_2 +\abs{I}^2  e^{-2\lambda\lambda't}\cS_K(f)^2
\ll
\left(\abs{I} e^{2t(\lambda'-1)} +  \abs{I}^{2}e^{-2\lambda\lambda't}\right)\cS_K(f)^2.
\]
%To obtain \equ{eq: norm estimate}, note that both terms on the right
%hand side are exponentially decreasing in $t$, but one term increases
%with $\lambda'$ and the 
%other decreases. 
The best choice is to take $\lambda'$ for which both
terms on the right hand side are equal, and this yields \equ{eq: norm estimate}. 
More precisely, equality of the two terms holds when we set:
\begin{equation}
\label{def:lambda'}
\lambda'=\frac{1}{\lambda+1}\left(\frac{1}{2t}\log\abs{I}+1\right). 
\end{equation}
With this choice, using $|I|<1 $ we find $0 < \lambda' <  \frac{\log\abs{I}}{2t}+1< 1$ provided $t > \frac{1}{2} \log \frac{1}{|I|} $,  and \equ{eq: norm estimate} holds.  
\end{proof}

The next result follows from 
the bound \equ{eq: norm estimate} combined with the
Borel-Cantelli Lemma and the Markov inequality.

\begin{theorem}
\label{thm:ergodic}
Let $(X,\mu)$ be a p.m.p.\ $G$-space with a spectral gap of
size $\lambda_X$. Let $\lambda < \lambda_X$, let $t_n \in \bR_+$, let
$\eta = \frac{1}{\lambda +1}$ and
let $\eta_1$ be such that 
\eq{eq:convergence condition}{
\sum_{n\in \bN} e^{-\lambda \eta_1 t_n}<\infty. 
}
Let $0 \leq \nu_n \leq 1 $ be a sequence of functions on $K$ as in
Theorem~\ref{thm:L2},  satisfying $ \nu_n(K) = \int \nu_n dm
_K>e^{-2t_n} $.  Let
$(f_{n})_{n \in \N}$ be a collection of functions in $\cS_{K,0}(X)$.
Then for almost all $x\in X$ there exists $n_0=n_0(x)$ such that if $n
\ge n_0$ then  
\eq{eq: ergodic estimate}{
\abs{\pi_{X}(\Sigma_{\nu_n,t_n})f_{n}(x)}
%\le e^{-\frac{\lambda}{4} t_n}\cS_K(f_{\delta_k,\eps_\ell})^{1/2}
\le e^{-(\eta-\frac{\eta_1}{2}) \lambda t_n
  } \cS_K(f_n) \nu_n(K)^{1-\frac{\lambda\eta}{2}}.
}
%(where $ \nu_n(K) = \int \nu_n dm_K>e^{-2t_n} $ and  
Here $\eta$  is as in \equ{eq: norm estimate}. 
\end{theorem}
Note that we will only be interested in the nontrivial case where the
right hand side of \equ{eq: ergodic estimate} decays with $t$,
i.e. when $\eta_1$
satisfies $0 < \frac{\eta_1}{2} < \eta$.

 \begin{proof}
 
Using \equ{eq: norm estimate}, there is $C>0$ such that 
for $f_n\in \cS_{K,0}(X)$  we have
\begin{equation}\label{effective-veech}
\norm{\pi_X(\Sigma_{\nu_n,t_n})f_n}^2_2\le Ce^{-2\lambda \eta t_n}C_n, \ \text{where } C_n =
 \cS_K(f_n)^{2} \nu_n(K)^{2-\lambda\eta}. 
\end{equation}

Consider 
%the sequence of
%operators $\Sigma_{\nu_n,t_n}$, $n\in \bN$, and consider 
for each $n$ the set of `bad points'
$$U_n = \set{x\,: \, e^{-\lambda\eta_1
    t_n/2}\abs{\pi_X(\Sigma_{\nu_n,t_n})f_n(x)} \ge e^{-\lambda \eta t_n}C_n^{1/2}}.$$
By Markov's inequality and (\ref{effective-veech}),
$$\mu(U_n)\le e^{-\lambda\eta_1
    t_n}\frac{\|\pi_X(\Sigma_{\nu_n,t_n})f_n\|_2^2}{
  e^{-2\lambda \eta t_n}C_n}\le C e^{-\lambda\eta_1 t_n}.$$ 
By \equ{eq:convergence condition} $\sum_{n\in \bN} \mu(U_n) < \infty$, so by
the Borel-Cantelli lemma, almost every $x\in X$ belongs to at most finitely many
of the sets $U_n$. We conclude that  
for almost every $x\in X$, there exists $n_0$ such that for all $n \ge
n_0$ we have $x\notin U_n$.
\end{proof}

Now let 
$$\pi_X\left(\Sigma_{\nu,t}^{(g)} \right)f (x) =
\pi_X(\Sigma_{\nu, t})f(gx) =  \int_K 
f(a_tkgx) \nu(k) dm_K(k)$$ denote the `dilated ellipse average' associated with
$g \in G$. 
%, let $\Psi_g:G \to G$ be the right multiplication map $h\mapsto
%hg$, and write 
%$$\|\Psi_g\| = \max \left( \|\Psi_g\|_{\mathrm{op}}, \| \Psi_g^{-1}
%  \|_{\mathrm{op}}^{-1} \right) .$$ 
%\bwpar{Revise def of $\| \Psi_g\|$. }
For the proof of Theorem \ref{thm: strengthening} we will 
need the following uniform versions of 
Theorems \ref{thm:L2} and \ref{thm:ergodic}: 

\begin{theorem}\label{thm:uniform ellipses}\name{thm: uniform ellipses}
With the notations of  Theorems \ref{thm:L2} and \ref{thm:ergodic},
for every $\lambda < \lambda_X$ there exists $C>0$ such that for all
$t>1$, any interval $I \subset \bS^1$ with $\abs{I}>e^{-2t}$, any $f \in \cS_{K,0}(X)$, and
any $g \in G$, we have 
\eq{eq: norm estimate ellipses}{
\left\| 
\pi_X\left(\Sigma^{(g)}_{\nu, t} \right)f  \right\|_2^2 \leq C e^{-2\lambda
  \eta t} %\| \Psi_g\|^{-\lambda \eta} 
\cS_K(f)^{2} |I|^{2-\lambda\eta}.
}
Furthermore, if $(t_n)\subset \bR_+$, $\eta_1>0$ satisfy
\equ{eq:convergence condition}, $0 \leq \nu_n \leq 1$ is a sequence of characteristic
functions on $K$ satisfying $\nu_n(K)>e^{-2t_n}$,  $(f_n)$ is a sequence of functions in $\cS_{K,0}$,
and $(g_n)$ is a countable subset of $G$,  then for almost all
$x \in X$ there is $n_0$ such that for all $n \geq n_0$ we have 
$$
\left|\pi_X\left(\Sigma^{(g_n)}_{t_n, \nu_n} \right)f (x) \right| \leq e^{-(\eta
  - \frac{\eta_1}{2})\lambda t_n } 
%\|\Psi_{g_n}\|^{-\frac{\eta}{2}} 
\cS_K(f_n) \nu_n(K)^{1-\frac{\lambda\eta}{2}}.
$$
\end{theorem}

\begin{proof}
A change of variables $y = gx$ shows that $\| \pi_X(\Sigma^{(g)})_{t,
  \nu} f \|_2 = \| \pi_X(\Sigma_{t,
  \nu}) f\|_2$ and thus \equ{eq: norm estimate ellipses} follows from the same argument as for
%is very similar to the proof of 
\equ{eq: norm estimate}. 
%We write the matrix coefficient as 
%$$
%\left\|\pi_X\left(\Sigma^{(g)}_{t_n, \nu_n} \right)f \right\|_2^2 = \langle
%\delta_{a_t} * \nu^{*2} * \delta_g* \delta_{a_{-t}} f, f \rangle.
%$$
%We define a parameter $\lambda'$ and intervals $I$ and $J$ as in the
%proof of Theorem \ref{thm:L2}, and when considering the
%integral on $I \sm J$, we note that $\|a_t  k g a_{-t}\| \geq \frac{
%e^{2\lambda' t}}{\|\Psi_g\|}.$ \bwpar{Fix and revise def of $J$} 
% This leads to 
%\[
%\|\pi_{X}(\Sigma_{\nu,t})f\|_2^2\ll
%|I| e^{2t(\lambda'-1)}\|f\|^2_2 +
%|I|^2 e^{-2\lambda\lambda't}\cS_K(f)^2\|\Psi_g\|^{-\lambda} ,
%\]
%and setting both terms on the right hand side equal to each other
%yields \equ{eq: norm estimate ellipses}. 

The proof of the second
assertion for any fixed choice of sequence $(g_n)$ is similar to the
proof of Theorem \ref{thm:ergodic}, 
using \equ{eq: norm estimate ellipses} instead of \equ{eq: norm estimate}. 
\end{proof}

\section{Control over the cusp}
The results of the previous section apply
to every action with a spectral gap, and we will want to apply them to
the action on the moduli space of flat surfaces, taking the
functions $f_n$ to be Siegel-Veech transforms of compactly supported
functions on $\bR^2$. However in this setting, the Sobolev norms
$\cS_K(f_n)$ might not be bounded, owing to a large contribution
coming from surfaces in the thin part, i.e. surfaces $\su$ with
$\ell(\su)$ small. When dealing with this issue it is helpful to note
that the Sobolev norm we have used above 
involves only differentiation in  the $K$-direction, and  
as we shall now see, this fact will allow us to use a simple argument
for  ``cutting off the cusp''.  
We let $M_\varepsilon=\ell^{-1}([\varepsilon,\infty))$. By a
well-known compactness criterion (see \cite[p.\ 152]{AGY}) the sets
$M_\vre$ 
are an exhaustion of $\cH$ by compact sets. 
Theorems \ref{thm:boundbyell} and \ref{thm:integrability} give bounds
on the measure of the complement $M^c_\vre = \cH \sm M_\vre$, and 
 on the time a translated circle spends in $M^c_\vre$. We will use these to 
cut off any function at the cusp without affecting its asymptotic
behavior. 

Before proceeding with this
argument, note that since we used the Euclidean metric in the
definition of the function $\ell$, the set $M_\varepsilon^{c}$ is
$K$-invariant,  and 
hence its characteristic function is $K$-smooth.  Below
we let $\partial_\theta$ denote the partial derivative in
  the spherical direction in polar coordinates. In terms of the action
  of $K$ on the plane, it is
  defined as $\pi_{\bR^2}(\omega)$ in the notation \equ{eq: pi
    notation}. Equivalently, at a point $\mathbf{y} \in \bR^2$,
$$
\partial_\theta \varphi (\mathbf{y}) = \left
. \frac{d}{d\phi}\right|_{\phi=0} \varphi(\exp(\phi\omega)\mathbf{y}).  
$$
%\rradd{Indeed, the relation of polar and euclidean coordinates is
%\[
%\partial_\theta=-y\partial_x+x\partial_y
%\]
%and for simplicity we may assume $\mathbf{y}=(0,r)$, so that $\partial_\theta=-y\partial_x$. It follows that
%\[
%\partial_\theta h(\mathbf{y}) =\lim_{t\to0}
%\frac{h(\exp({t\omega})\mathbf{y})-h(\mathbf{y})}{t}=\lim_{t\to0}
%\frac{h(\mathbf{y}+t\omega \mathbf{y})-h(\mathbf{y})}{t} 
%\]
%\[
%=\lim_{t\to0} \frac{h((tr,r))-h((0,r))}{t}=\lim_{tr\to0}
%r\frac{h((tr,r))-h((0,r))}{tr}=r\partial_xh(\mathbf{y})=\partial_\theta
%h(\mathbf{y}) 
%\]
%}
%%; in other words, at the
%%  point $(x,y) =
%%  r(\cos \theta, \sin \theta)$, it is defined by
%%  $\partial_\theta f(x,y) = \lim_{t \to 0} \frac{f(x - tr\sin \theta,y
% %   + tr\cos \theta) - f(x,y)}{t}.$ 
%\bwpar{make sure this def of $\partial_\theta$ corresponds to what you had in mind.} 
\begin{lemma}
\label{cutoff}
Suppose $R>0$ and $\psi: \bR^2 \to \bR$ is a 
non-negative bounded function which is supported in the ball $B(0,R)$,
such that $\partial_\theta \psi$ is also bounded, 
and denote by $f=\widehat{\psi}$ its Siegel-Veech transform as in
Theorem \ref{thm:siegelveech}, with respect to some configuration
$\cC$. 
%Let $\nu$ be the indicator function of any interval $I \subset
%\mathbb{S}^1 \cong K$. 
Let $\chi_\varepsilon$ denote the
characteristic function of the cusp 
$M_\varepsilon^c$.
%, and as before write $\|B\|_\infty = \sup\{\|x\|: x
%\in B\}.$ 
Then the decomposition 
\[
f=f_{\on{main}}+f_\varepsilon, \ \ \text{ where } f_{\on{main}}=
f(1-\chi_\varepsilon) \ \text{ and } f_\varepsilon = f\chi_\varepsilon
\]
satisfies for any $1<\ref{exp:boundbyell}<\ref{exp:integrability}<2$,
\eq{eq: lem 1}{
\cS_K(f_{\on{main}})^2\ll_{R, \alpha_1} \max_{r \leq R} 
\left(\int_K \left(\psi^2 + |\partial \psi|^2 \right) (r k\mathbf{e}_1)
  dm_K (k)\right) \vre^{-2\alpha_1},
}
\eq{eq: lem 2}{
\int_\cL f_\varepsilon \, d\mu
\ll_{R, \alpha_1, \alpha_2} \|\psi\|_\infty
\varepsilon^{\ref{exp:integrability}-\ref{exp:boundbyell}}, 
}
and
\eq{eq: lem 3}{
 \pi_{\cL}(\Sigma_{%\nu, 
t} )f_\varepsilon (\su) 
\ll_{\su, R, \alpha_1, \alpha_2} \|\psi\|_\infty
\varepsilon^{\ref{exp:integrability}-\ref{exp:boundbyell}}. 
}
Moreover the implicit constant in \equ{eq: lem 3} can be taken to be
uniform as $\su$ ranges over compact subsets of $\cL$. 
\end{lemma}

\begin{proof} 
We first bound the $L^2$-norm of $f_{\mathrm{main}}$. Since the measure $\mu$ and the set 
$M_\vre$ are $K$-invariant, and $\mathbf{y}\mapsto V(\mathbf{y})$ is $K$-equivariant, we have
\[
\begin{split}
\|f_{\on{main}}\|_2^2 = &
\int_{\cL} |f(1-\chi_\varepsilon)|^2d\mu 
=
\int_{\ell(\mathbf{y}) \geq\varepsilon} \left|\sum_{v \in V(\mathbf{y})}
  \psi(v)\right|^2 d\mu (\mathbf{y})
\\
= & \int_{\ell(\mathbf{y}) \geq\varepsilon}\int_K \left|\sum_{v \in V(k\mathbf{y})}
  \psi(v)\right|^2 dm_K (k)\, d\mu (\mathbf{y}) \\
\leq &\int_{\ell(\mathbf{y}) \geq\varepsilon}  \left| V(\mathbf{y})\cap B(0,R)\right| \ \sum_{v \in V(\mathbf{y}) \cap B(0, R)} \int_K |\psi(kv)|^2
dm_K(k) \, d\mu(\mathbf{y})\\
%\leq &
%\int_{\ell(\mathbf{y})(\su)\geq\varepsilon}\int_K |V(kx)\cap \supp\psi|\sum_{v \in V(\su)}
%  \left|\psi(kv)\right|^2 dm_K (k) \, d\mu \\ 
\leq & \int_{\ell(\mathbf{y}) \geq\varepsilon}  |V(\mathbf{y})\cap B(0,R)|^2 \ 
  \left(\max_{r\leq R}\int_K\left|\psi(rk \mathbf{e}_1)\right|^2 dm_K
  (k) \right)  \, d\mu (\mathbf{y}). 
%\\
%= &\max_{r \leq R} \int_K\left|\psi(rk_\theta \mathbf{e}_1)\right|^2
% dm_K (k)\int_{R\geq\ell(\su)\geq\varepsilon} |V(\su)\cap B(0, R)|^2
%  d\mu .
%\\ 
% \stackrel{\text{Thm. }\ref{thm:boundbyell}}{\ll_{R, \alpha_1} }& \max_{r
%   \leq R} \int_K\left|\psi(rk_\theta \mathbf{e}_1)\right|^2 
% dm_K (k) 
%\varepsilon^{-2\ref{exp:boundbyell}}. 
\end{split}
\]
In the first inequality above we have used Cauchy-Schwarz to get an
estimate $\left|\sum_{v \in V(k\mathbf{y})}
  \psi(v)\right|^2\le \left(\sum_{v \in V(k\mathbf{y})}
  |\psi(v)|^2\right) |V(\mathbf{y})\cap B(0,R)|$, and then exchanged
summation and integration. For the second inequality, note that 
each $v \in V(\mathbf{y}) \cap B(0,R)$ we
can write $v = r k \mathbf{e}_1$ for some rotation $k=k(v) \in K$ and some positive scalar $r=r(v)
\leq R$, and the estimate follows. 
Using Theorem \ref{thm:boundbyell}, we conclude that
\eq{ineq:boundingmain}{
	\|f_{\on{main}}\|_2^2\ll_{R, \alpha_1} 
        \varepsilon^{-2\ref{exp:boundbyell}}\max_{r
          \leq R}\int_{ K}|\psi(rk \mathbf{e}_1)|^2 dm_K(k)
}

We repeat this calculation for the angular derivative of $\psi$. 
Here we also use the fact that since the set of saddle connections satisfies
$V(g\mathbf{y})=gV(\mathbf{y})$ for any $\mathbf{y} \in \cL$ and $g
\in G$,  which implies that 
taking derivatives in the $K$ direction commutes with the Siegel-Veech
transform. Namely, for any compactly supported  $\psi:\bR^2\to\bR$ for which
$\partial_\theta \psi$ exists everywhere,
\[
\pi_{\cL}(\omega)\widehat{\psi}=\lim_{\phi\to
  0}\frac{1}{\phi}\left(\pi_{\cL}(\exp(\phi\omega)\widehat{\psi}-\widehat{\psi}\right)
=\lim_{\phi\to 0}\frac{1}{\phi}\left(\sum_{u\in
  V(\exp(\phi\omega)\mathbf{y})}\psi(u)-\sum_{v\in V(\mathbf{y})}\psi(v)\right) 
\]
\[
=\lim_{\phi\to 0}\frac{1}{\phi}\left(\sum_{v\in
  V(\mathbf{y})}\left(\psi(\exp(\phi\omega)v)-\psi(v)\right)\right) 
=\sum_{v\in V(\mathbf{y})}\partial_\theta \psi(v)
\]
(where we have used the fact that $\psi$ is compactly supported to
ensure that the sum is finite and hence we can switch the order of
summation and 
differentiation). Thus 
$$
\pi_{\cL}(\omega) \hat{\psi}= \widehat{\partial_\theta \psi}.
$$
By $K$-invariance of $(1-\chi_\varepsilon)$,  
\[
\pi_{\cL}(\omega)f_{\on{main}}=(1-\chi_\varepsilon)\pi_{\cL}(\omega)f+
f\pi_{\cL}(\omega)(1-\chi_\varepsilon) 
= (1-\chi_\varepsilon)\widehat{\partial_\theta \psi}+0
\]
and consequently, applying the  argument used to prove inequality
\equ{ineq:boundingmain} 
to $\pi_{\cL}(\omega)f_{\on{main}}$, we obtain \equ{eq: lem 1}.
\ignore{
\[
\cS_K(f_{\on{main}})^2=\|f_{\on{main}}\|_2^2+\|\pi_{\cL}(\omega)f_{\on{main}}\|_2^2
\ll_{B, \ref{exp:boundbyell}} (\|\psi\|_\infty^2+\|\partial_\theta
\psi\|_\infty^2)\varepsilon^{-2\ref{exp:boundbyell}}. 
\]}

Now we set $\beta = \alpha_2 - \alpha_1$, and proceed to bound 
$\int f_\varepsilon \, d \sigma$ in the two cases 
$\sigma=\mu, \ \sigma=\pi_{\cL}(\Sigma_t)=\int_K \delta_{a_tk\su}
dm_K(k)$.
We then have
% By
%Theorem~\ref{thm:boundbyell} for the second inequality in  
\[
\int f_\varepsilon \, d \sigma =\int_{\ell(\mathbf{y})<\varepsilon}\widehat{\psi}
d\sigma\leq \int_{\ell(\mathbf{y})<\varepsilon}\|\psi\|_\infty |V(\mathbf{y})\cap
B(0,R)| d\sigma 
\]
\[
\stackrel{\text{(Thm.~\ref{thm:boundbyell}})}{\ll_{R, \alpha_1}} \|\psi\|_\infty\int_{\ell(\mathbf{y})<
  \varepsilon}\ell(\mathbf{y})^{-\ref{exp:boundbyell}}d\sigma 
< \|\psi\|_\infty\int_{\ell(\mathbf{y})<
  \varepsilon}\frac{\varepsilon^\beta}{\ell(\mathbf{y})^\beta}\ell(\mathbf{y})^{-\ref{exp:boundbyell}}d\sigma. 
\]
The term $\int_{\ell(\mathbf{y})<
  \varepsilon}\ell(\mathbf{y})^{-\ref{exp:boundbyell}-\beta}d\sigma$ is bounded
by $\|\ell(\cdot)^{-\alpha_2}\|_{L^1(\sigma)}$ for
$\ref{exp:integrability}=\ref{exp:boundbyell}+\beta$, 
which in turn is bounded by Theorem~\ref{thm:integrability} for any
$\ref{exp:integrability}<2$ (where 
for $\sigma = \pi_{\cL}(\Sigma_t) $ the bound depends on 
$\su$ uniformly on compact subsets of $\cL$, and is independent of $t$). 
\end{proof}

%We also note the following immediate consequence of Theorem
%\ref{thm:integrability}, which will be useful in the proof of Theorem
%\ref{thm: strengthening}: 

%\begin{theorem}\name{thm: integrability ellipses}
%\label{thm: integrability ellipses}
%For any $\su\in \cL$, and for any $1\leq\ref{exp:integrability}<2$,
%\begin{equation}\label{eq: integrability ellipse}
%\sup_{t>0} \pi_\cL\left(\Sigma^{(g)}_t
%\right)\left(\ell(\su)^{-\ref{exp:integrability}}\right)<\infty\, 
%\end{equation}
%with a uniform bound for $g$ varying over a compact set in $G$.
%\end{theorem}
%\begin{proof}
%Let $g\in C$ where $C\subset G$ is compact. Then
%\[
%\int_K 
%\ell^{-\alpha_2}(a_tkgx) dm_K=\int_K 
%\ell^{-\alpha_2}(a_tkgx)\leq\sup_{y\in Cx} \int_K
%\ell^{-\alpha_2}(a_tky) dm_K,  
%\]
%and the claim follows from
%Theorem \ref{thm:integrability}. 
%\end{proof}

\ignore{Setting $\beta = \alpha_2 - \alpha_1$, we proceed to the bound for
$\int f_\varepsilon \, d \sigma$ for the two cases 
$\sigma=\mu, \ \sigma=\eta_{t,\su}$. Using Theorem~\ref{thm:boundbyell}  
\[
\int f_\varepsilon \, d \sigma =\int_{\ell(\su)<\varepsilon}\widehat{\psi}
d\sigma\leq \int_{\ell(\su)<\varepsilon}\|\psi\|_\infty |V(\su)\cap
B| d\sigma 
\]
\[
\ll_{B, \alpha_1} \|\psi\|_\infty\int_{\ell(\su)<
  \varepsilon}\ell(\su)^{-\ref{exp:boundbyell}}d\sigma 
< \|\psi\|_\infty\int_{\ell(\su)<
  \varepsilon}\frac{\varepsilon^\beta}{\ell(\su)^\beta}\ell(\su)^{-\ref{exp:boundbyell}}d\sigma. 
\]
The term $\int_{\ell(\su)<
  \varepsilon}\ell(\su)^{-\ref{exp:boundbyell}-\beta}d\sigma$ is bounded
by $\|\ell(\cdot)^{-\alpha_2}\|_{L^1(\sigma)}$ for
$\ref{exp:integrability}=\ref{exp:boundbyell}+\beta$, 
which in turn is bounded by Theorem~\ref{thm:integrability} for any
$\ref{exp:integrability}<2$ (where 
for $\sigma = \eta_{t, \su}$ the bound depends on 
$\su$ but is independent of $t$). 
\end{proof}
}

\ignore{
For the proof of Theorem \ref{thm: strengthening}, it will be useful
to have more information about the dependence of the implicit
constants in \equ{eq: lem 1}, \equ{eq: lem 2} and \equ{eq: lem 3} on
the compact set $B$. In this 
direction we have:
\begin{lemma}\name{lem: cutoff2}
With the notation and assumptions of Lemma \ref{cutoff}, we have: 

\eq{eq: lem 1.2}{
\cS_K(f_{\on{main}})\ll_{\ref{exp:boundbyell}} \|B\|_\infty^{\alpha_1}
(\|\psi\|_\infty^2+\|\partial_\theta
\psi\|_\infty^2)^{\tfrac12}\varepsilon^{-\ref{exp:boundbyell}}, 
}
\eq{eq: lem 2.2}{
\int f_\varepsilon \, d\mu \ll_{\ref{exp:boundbyell},\ref{exp:integrability}} \|B\|_\infty^{\alpha_1}
\|\psi\|_\infty\varepsilon^{\ref{exp:integrability}-\ref{exp:boundbyell}},
}
\eq{eq: lem 3.2}{
\int f_\varepsilon \, d\eta_{t,\su}
\ll_{\su, \ref{exp:boundbyell},\ref{exp:integrability}} \|B\|_\infty^{\alpha_1}
\|\psi\|_\infty\varepsilon^{\ref{exp:integrability}-\ref{exp:boundbyell}}. 
}\end{lemma}

\begin{proof}

\end{proof}
}
\section{Effective counting of saddle connections}\name{sec: effective counting}
We now give the proof of Theorem \ref{thm:theorem}, dividing the
argument into three steps. In the first we use a geometric
counting method introduced in \cite[Lem. 3.6]{EsMaMo} and
\cite[Lem. 3.4]{EsMa} to estimate the quantity $N(e^t, \su, \varphi_1,
\varphi_2)$ by orbit integrals $\pi_\cL(\Sigma_t)f(\su)$, where $ f$ is a Siegel-Veech 
transform of the indicator of a
triangle. We will follow the simplified approach 
outlined in the survey \cite{eskin2006counting}, but replacing a trapezoid used in 
\cite{eskin2006counting} with a triangle. In the second step we will
replace $f$ with certain smooth approximations and use
Theorem \ref{thm:ergodic} and Lemma \ref{cutoff} to  estimate the
resulting orbit integrals. Since it relies on the Borel-Cantelli
lemma, Theorem \ref{thm:ergodic} only gives information 
about $N(T, \su, \varphi_1,
\varphi_2)$ for a countable number of values of $T$. In the third and
final step we use an interpolation argument to pass from countably
many values, which get denser and denser on a logarithmic scale, to
all $T$.

%\begin{proof}[{Proof of Theorem~\ref{thm:theorem}}] 
%\item
\medskip 

\paragraph{\textbf{Step 1. Triangles,  and reduction of counting to
    orbit integrals}}
We fix a configuration $\cC$ and use it to define a Siegel-Veech
transform as in Theorem \ref{thm:siegelveech}. 
For $\theta \in (0,1)$ we define two triangles 
$W_1=W_1(\theta)$ and $W_2=W_2(\theta)$ in the plane as follows. Let $e_2 =
(0,1)$ and let $W_1$ have vertices $(0,0),
r_\theta e_2, r_{-\theta}e_2$ and $W_2$ have vertices $(0,0),
\frac{1}{\cos\theta} r_\theta e_2,
\frac{1}{\cos\theta} r_{-\theta}e_2$. That is, $W_1$ and $W_2$ are
similar isosceles 
triangles with apex at the origin, apex angle $2 \theta$, symmetric
around the positive $y$-axis, and with height $\cos\theta$ and 1
respectively. In particular $W_1 \subset W_2$. See Figure \ref{fig: 1}. 

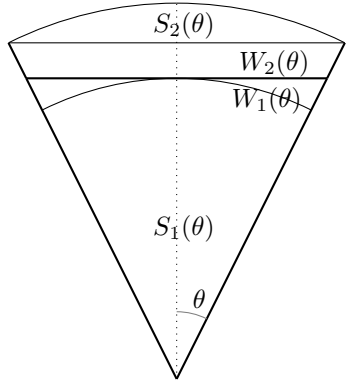
\begin{figure}
\begin{tikzpicture}
\draw [thick] (0,0) --(2,4);
\draw [thick] (0,0) --(-2,4);
\draw [thick] (-2,4) --(2,4);
\draw [thick] (2.236,4.472) --(2,4);
\draw [thick] (-2.236, 4.472) --(-2,4);
\draw [dotted] (0,0) --(0, 5);
\draw (-2.236, 4.472) --(2.236, 4.472);
\draw [gray] (0, 0.8944) arc [radius=0.8944, start angle=90, end
angle=63.5];
\draw  (0, 4) arc [radius=4, start angle=90, end
angle=63.5];
\draw  (0, 4) arc [radius=4, start angle=90, end
angle=116.5];
\draw  (0, 5) arc [radius=5, start angle=90, end
angle=63.5];
\draw  (0, 5) arc [radius=5, start angle=90, end
angle=116.5];
\node at (0.3, 1.06) {$\theta$};
\node at (0.1, 2) {$S_1(\theta)$};
\node at (0.1, 4.7) {$S_2( \theta)$};
\node at (1.2, 3.7) {$W_1(\theta)$};
\node at (1.3, 4.2) {$W_2(\theta)$};
%\draw (2,4) arc [radius=15, starting angle=94, end angle=92];
\end{tikzpicture}
\caption{The four planar domains $S_1(\theta) \subset W_1(\theta)
  \subset W_2(\theta) \subset S_2(\theta)$.}
\label{fig: 1}
\end{figure}
 
Now let $t>1$ be a parameter.
Applying the diagonal flow $a_{-t}$ transforms $W_1, W_2$ into triangles
with a narrow apex angle and large height, specifically the apex angle
$2\theta_t$ of both $a_{-t}W_1$ and $a_{-t}W_2$ satisfies
\eq{eq: angle satisfies}{
\tan\theta_t=e^{-2t}\tan\theta.
}
We will obtain lower and upper bounds for
$N^{\cC}(e^t,\su,\varphi_1,\varphi_2)$ using radial averages over shrinking versions of these
triangles.

Let $\varphi_1<\varphi_2$ be as in Theorem
\ref{thm:theorem}. By a
rotation, assume with no loss of generality that $\varphi_2 =
\varphi>0$ and $\varphi_1 =
-\varphi$ so that $I=[-\varphi,\varphi]$ is
symmetric around 0 and $\varphi_2 - \varphi_1 = 2\varphi$. Recall the
notation $r_s$ for an element of $K$ (see \equ{eq: consider}). 
%, and for
%an interval $J \subset \bR$, define 
%define $K(J) = \{r_s : s \in J\}$. Thus $K(J) = K(J +2\pi
%k)$ for any $k \in \bZ$, and  $K(I) = K(-I)$ since $I$ is symmetric
%around 0. 
We will identify angles in $\bR$ with their image modulo
$2\pi \bZ$ and  functions on $K$ with functions on $\bR/2\pi
\bZ$ without further mention. 

Define 
\eq{eq: in analogy}{
I^-_t = [-(\varphi-\theta_t),\varphi-\theta_t], \ \
I^+_t = [-(\varphi+\theta_t),\varphi+\theta_t],
}
so that $I^-_t \subset I \subset I^+_t$, and let $\nu^-_t, \nu, \nu^+_t$
denote respectively the measures whose densities are the indicator functions of  
$I^-_t, I, I^+_t$ (note that the dependence of these indicators on
$\varphi_2 - \varphi_1$ is 
suppressed from the notation).
Also let 
%$\nu^-$ denote the indicator of
%$[-\theta_t, \theta_t]$ and let 
$\mathbbm{1}_{W_1}, \mathbbm{1}_{W_2} $
denote the indicators of $W_1$ and $W_2$.

We claim that for any $\su$, 
\eq{eq: what we claim}{
\pi_{\cL}(\Sigma_{\nu^-_t, t})\widehat{\mathbbm{1}_{W_1(\theta)}}(\su)
\leq \frac{\theta_t}{\pi}
N^{\cC}(e^{t},\su,\varphi_1,\varphi_2) \leq 
\pi_{\cL}(\Sigma_{\nu^+_t, t})\widehat{\mathbbm{1}_{W_2(\theta)}}(\su). 
}
To see the left hand inequality, recall that by the definition of the
Siegel-Veech transform and the operator $\Sigma_{\nu, t}$ we have 
\eq{eq: this sum}{
\pi_{\cL}(\Sigma_{\nu^-_t, t} )\widehat{\mathbbm{1}_{W_1(\theta)}}(\su)
= \sum_{v \in V^{\cC}(\su)} \int_K \mathbbm{1}_{W_1(\theta)}(a_tkv) \nu^-_t(k)
dm_K(k).
}
We will estimate the contribution of each individual $v \in V^{\cC}(\su)$ to
the sum \equ{eq: this sum}. For any $v\in \bR^2$,
\eq{eq: the contribution}{
\int_K\mathbbm{1}_{W_1}(a_tkv) \nu_t^-(k)dm_K(k) = \frac{1}{2\pi}
\int_{-\varphi+\theta_t}^{\varphi - \theta_t} 
\mathbbm{1}_{a_{-t}W_1}(k_\phi v) d\phi
}
is at most $\frac{\theta_t}{\pi}$, since the apex angle of $a_{-t}W_1$
is $2 \theta_t$. The quantity \equ{eq: the contribution} vanishes 
if $\|v\|\geq e^t$ or $\measuredangle(v,e_2)\notin
I$, since in these cases the arc $K(I_t^-)v = \{k_\beta v: \beta \in I_t^-\}$ never enters the
triangle $a_{-t}W_1$ (see Figure \ref{fig: 2}).
\begin{figure}
\begin{tikzpicture}
\draw [thick] (0,0) --(0.1,4);
\draw [thick] (0,0) --(-0.1,4);
\draw [thick] (-0.1,4) --(0.1,4);
\draw  (0, 2.5) arc [radius=2.5, start angle=90, end
angle=60];
\draw  (0, 2.5) arc [radius=2.5, start angle=90, end
angle=100];
\draw  (0, 5.5) arc [radius=5.5, start angle=90, end 
angle=75];
\draw  (0, 5.5) arc [radius=5.5, start angle=90, end
angle=115];
\draw  (-0.607 , 3.446 ) arc [radius=3.5, start angle=100, end 
angle=140];
\draw  [dotted, ->] (0.5 , 3.6 ) arc [radius=4, start angle=85, end 
angle=92];
\node at (0.7, 2.7) {$v_1$};
\node at (-0.6, 5.2) {$v_2$};
\node at (1.2, 3.6) {$a_{-t} W_1(\theta)$};
\node at (-2, 3.3) {$v_3$};
\draw[fill] (0.434,2.462) circle [radius = 0.025];
\draw[fill] (-0.479,5.475) circle [radius = 0.025];
\draw[fill] (-1.75,3.031) circle [radius = 0.025];
%\draw (2,4) arc [radius=15, starting angle=94, end angle=92];
\end{tikzpicture}
\caption{The arc  $ K(I_t^-)v_1$ cuts through  $a_{-t} W_1$ 
  but the arcs $K(I_t^-)v_2 $ %(since $\|v_2 \|> e^t$) 
and $K(I_t^-)v_3
  $ 
%(since $\measuredangle(v_3 , e_2) \notin I$ ) 
miss $a_{-t}
  W_1(\theta)$. }
\label{fig: 2}
\end{figure}
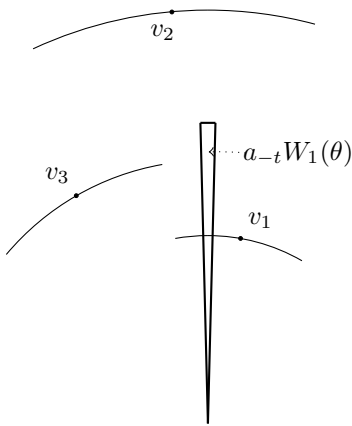
Furthermore, if 
$\|v\|\leq e^t$ and $\measuredangle(v,e_2)\in
I$ then the arc $K(I^+_t)v$ intersects $a_{-t}W_2$ along its  entire
apex angle, and so 
$
\int_K\mathbbm{1}_{W_2}(a_tkv) \nu_t^+(k)dm_K = \frac{\theta_t}{\pi},
$
and this implies the right hand inequality. 

\ignore{
. Indeed, one may rewrite
\[
\int_K\mathbbm{1}_{W_1}(a_tkv) \nu_t^-(k)dm_K=\int_{K(I^-_t)}
\mathbbm{1}_{ a_{-t}W_1}(kv)dk. 
\]
Since the apex angle of $a_{-t}W_1$ is $2\theta_t$, if $\|v\|\leq e_t\cos\theta$, then the
orbit $Kv$ will spend $2\theta_t/2\pi$ of the time 
inside the triangle, and the same holds true for $K(I_t)v$ if $v$ has
angle at most $|I|/2-\theta_t$ from the positive vertical axis. On the other hand, if this
angle is larger than $|I|/2$ then $K(I_t)v$ is disjoint from
$a_{-t}W_1$. 
We thus conclude a lower bound,
\[
\sum_{v\in V(\su)}\int_K\mathbbm{1}_{W_1}(a_tkv)
\nu^-_t(k)dm_K\leq \frac{2\theta_t}{2\pi} N(e^t,\su,\varphi_1,\varphi_2). 
\]
For the upper bound, we let $\nu_t^+(k)$ be the analogous 
characteristic function over $k_\varphi K(I^+_t)$ where now 
$I^+_t=[\varphi_1-\theta_t,\varphi_2+\theta_t]$
is the thickened interval. In this case, arguing as above, 
\[
\int_K\mathbbm{1}_{W_1}(a_tkv) \nu^+(k)dm_K
\]
contributes precisely $2\theta_t/2\pi$ for all $v$ with $\|v\|\leq
e^t$ and $\measuredangle(v,e_2)\in [\varphi_1,\varphi_2]$ so that
\bwpar{This used to say $w$ but probably you mean $W$.}
\[
\frac{2\theta_t}{2\pi} N(e^t,\su,\varphi_1,\varphi_2)\leq \sum_{v\in
  V(\su)}\int_K\mathbbm{1}_{W_2}(a_tkv) \nu^+_t(k)dm_K. 
\]  
}

%\item
\medskip 

\paragraph{\textbf{Step 2. Smooth approximations, ergodic theorem, and
    cutting off the cusp.}}
Our goal will be to estimate the left and right hand sides of \equ{eq:
  what we claim}. To this end we will replace 
$\widehat{\mathbbm{1}_{W_1}}, \widehat{\mathbbm{1}_{W_2}}$ with smooth
approximations and
  apply Theorem \ref{thm:ergodic} and Lemma \ref{cutoff} to the
  approximating functions, where the
  $n$-th function will give a bound on \equ{eq: what we claim} for a
  certain time $T_n=e^{t_n}$. Our approximation depends on functions and
  parameters which we now describe (omitting their dependence on
  $n$). The first parameter 
  $\theta$ controls the apex angle of the triangles $W_i$, as above. We
  will bound $\mathbbm{1}_{W_i}$ from above and below by $K$-smooth
  functions $\psi_{(-,\delta)}, \psi_{(+, \delta)}$  on the plane,
  supported respectively in a slightly contracted (resp. expanded) copy of
  $W_1$ (resp. $ W_2$), where the dilation is controlled by a smoothing parameter
  $\delta$. The corresponding Siegel-Veech
  transforms $\widehat{\psi_{(\pm, \delta)}}$ will be denoted by  $f_{(\pm,
    \delta)}$. They 
  will be truncated using Lemma \ref{cutoff} along with a cutoff
  parameter $\vre$,  for an appropriate choice of parameters
  $\alpha_1, \alpha_2. $ Finally Theorem \ref{thm:ergodic} will be
  applied to the main term $(f_{(\pm, \delta)})_{\mathrm{main}}$, and
  all resulting errors will be collected and bounded. 

We now make this discussion more precise and record some estimates for
the errors incurred at the various stages. After collecting these
bounds we will choose our parameters and optimize the error terms in
the next step. Our optimization gives $\theta = \delta^{1/2}$, and so
in order to reduce the number of parameters we will use this
dependence of $\theta$ and $\delta$ throughout. 

For $\theta < 1$ we approximate both triangles $W_1(\theta), W_2(\theta)$ by sectors
around the positive vertical axis,
that is by sets of the form 
$$S_{r_0, \varphi_0} = \{r(\cos \beta, \sin \beta): 0 \leq r \leq
r_0, |\beta - \pi/2|\leq \varphi_0\}.$$ 
The sector $S_1(\theta) = S_{\cos \theta, \theta}$ is contained in $W_1(\theta)$ and
the sector $S_2(\theta) = S_{(\cos \theta)^{-1}, \theta}$ contains
$W_2(\theta)$ (see Figure \ref{fig: 1}). Let
  $\delta =\theta^2$  and let
$H: \bR \to \bR $ be a continuously differentiable function which vanishes
outside $[-\theta, \theta]$, is equal to 1 on
$[-\theta+\delta, \theta -\delta]$, and such that $\|H'\|_\infty \leq
2\delta^{-1}$, and define 
\eq{eq: def psi}{
\psi_{(-,\delta)} (r \cos \beta, r \sin \beta) =
\left\{ \begin{matrix} H(\beta - \pi/2) & r \leq \cos \theta
\\ 0 & r> \cos \theta \end{matrix} 
\right.  }
We have pointwise inequalities
$$
\psi_{(-, \delta)} \leq \mathbbm{1}_{S_1(\theta)} \leq \mathbbm{1}_{W_1(\theta)},
$$
and an estimate
$$\|\partial_\theta \psi_{(-, \delta)} \|_\infty\ll \delta^{-1}.$$
Similarly, we define functions $\psi_{(+, \delta)}$ which satisfy a
pointwise inequality $\mathbbm{1}_{W_2(\theta)} \leq \psi_{(+,
  \delta)}$ and also satisfy 
$\|\partial_\theta \psi_{(+, \delta)} \|_\infty\ll \delta^{-1}.$
Since $\partial_\theta \psi_{(\pm, \delta)}(kr\mathbf{e}_1)$ is only supported on an
arc of angular width $2\delta$, this implies that 
\eq{eq: previous paragraph}{
\max_{r} \left(\int_K |\partial_\theta \psi_{(\pm,
    \delta)}(rk\mathbf{e}_1)|^2 dm_K(k) \right)  
\ll \delta^{-1}.
}
Since we also have
pointwise bounds $\mathbbm{1}_{S_1(\theta - \delta)} \leq
  \psi_{(-, \delta)} \leq \psi_{(+, \delta)} \leq
  \mathbbm{1}_{S_2(\theta+\delta)}$, we obtain a bound 
\eq{eq: obtain a bound}{
\int_{\R^2} (\psi_{(+,\delta)} - \psi_{(-,\delta)}) dx \leq \mathrm{Area}
\left(S_2(\theta+\delta) \sm S_1(\theta - \delta) 
\right) \ll \delta + \theta^3 \ll \delta.
}
Similarly, we obtain the bounds 
\begin{equation}
\label{ineq:approxbysmooth}
\int_{\bR^2}(\mathbbm{1}_{W_1(\theta)}-\psi_{(-,\delta)}) dx\ll \delta \; \text{ and }
\int_{\bR^2}(\psi_{(+,\delta)}-\mathbbm{1}_{W_2(\theta)}) dx\ll \delta.
\end{equation}
%\item
%\paragraph{\textbf{Area of a triangle}}
%We will have to estimate the expression 
%\[
%\left(\int_{\bR^2}\psi_{(\pm,\delta)}
%dx \right) \, \left(\int_K\nu_{t}^\pm dm_K\right).
%\]
%in \eqref{eq: final f}.
%, or rather its quotient with $\theta_t e^{2t}$, in perspective of \equ{eq: what we claim}, as we wish to estimate \(\frac{N^{\cC}(e^{t},\su,\varphi_1,\varphi_2)}{e^{2t}}\).
Since 
$$\mathrm{Area} (
%\int_{\bR^2}\mathbbm{1}_{
W_1(\theta)
)
%}
=\cos\theta\sin\theta \ \ 
 \text{
  and } \ \mathrm{Area} (
%\int_{\bR^2}\mathbbm{1}_{
W_2(\theta))
%}
=\tan\theta,$$ 
\eqref{ineq:approxbysmooth} also implies that 
\begin{equation}
\label{ineq:approxbysmooth2}
\left | \int_{\bR^2} \psi_{(-, \delta)} dx - \cos \theta \sin
  \theta \right| \ll \delta, \ \  \left |
  \int_{\bR^2} \psi_{(+, \delta)} dx - \tan \theta \right| \ll  \delta.
\end{equation}
We introduce the notation 
\[
\tilde{\bf{\theta}} \stackrel{\text{Def}}{=} e^{2t}\theta_t
=\theta+O(\theta^{3}).
\]
 Using
\equ{eq: angle satisfies} and expanding the Taylor series for  $\sin$, $\cos$,
$\arctan$ we find
$$
\frac{\sin\theta\cos\theta}{\tilde{\bf{\theta}}}=\frac{\theta+O(\theta^{3})}{\theta+
  O(\theta^{3})}=1+O(\theta^{2}), \ \ \text{ and also }
\frac{\tan\theta}{\tilde{\bf{\theta}} }= 1+O(\theta^2),
$$
so that 
\begin{equation}
\label{eq:maintermpsi}\tilde{\textbf{1}}\stackrel{\on{Def}}{=}
\frac{1}{\tilde{\bf{\theta}} }\int_{\bR^2} \psi_{(\pm, \delta)} dx=
1+\frac{1}{\tilde{\bf{\theta}}{}
}O(\theta^2+\delta)=1+O(\delta^{1/2}). 
\end{equation}
Note that the appearance of 
$\theta^2 + \delta$ in this bound explains our choice $\theta^2 =
\delta$. Note also that both $\tilde \theta$ and $\tilde{\textbf{1}}$
depend on $t$, but this dependence is suppressed from the notation. 
  
Provided $\theta \leq\pi/8$, and
recalling from \equ{eq: in analogy}  that
$\nu_t$ is supported on $I^\pm_t$, we have
%\eq{eq: recall}{
\[
\left| \int_K\nu_{t}^\pm dm_K-
\frac{\varphi_2-\varphi_1}{2\pi} \right| \leq \frac{\theta_t}{\pi}\leq \frac{\tan \theta_t}{\pi}\leq e^{-2t}\tan \theta\le e^{-2t}.
%}
\]
We will make our choices so that 
\eq{eq: one more requirement}{
e^{-2t} = o(\delta^{1/2}),
}
so that \eqref{eq:maintermpsi} implies 
\eq{eq:main term bound}{\left|\tilde{\textbf{I}} \right|=
\frac{\varphi_2-\varphi_1}{2\pi}+O(\delta^{1/2}), \ \ \text{ where } 
\tilde{\textbf{I}}\stackrel{\on{Def}}{=}\tilde{\textbf{1}}
\int_K\nu_{t}^\pm dm_K. 
}
%\item
%\paragraph{\textbf{Function Decomposition and Estimating Sobolev norms}}
Let $f_{(\pm, \delta)} = \widehat{\psi_{(\pm, \delta)}}$ be the
Siegel-Veech transform of the functions defined in \equ{eq:
  def psi}. The transform preserves pointwise inequalities of
functions, and so \equ{eq: what we claim} implies
\eq{eq: what we claim2}{\frac{\pi}{\theta_t}
\pi_{\cL}(\Sigma_{\nu^-_t, t})f_{(-, \delta)}(\su)
\leq N^{\cC}(e^t,\su,\varphi_1,\varphi_2) \leq \frac{\pi}{\theta_t}
\pi_{\cL}(\Sigma_{\nu^+_t, t})f_{(+, \delta)}(\su).
}
We apply Lemma \ref{cutoff} with parameters $\alpha_1, \alpha_2,
\vre$. Decompose $f_{(\pm,\delta)}$ as the sum $(f_{(\pm,
  \delta)})_{\mathrm{main}} +   (f_{(\pm,
  \delta)})_{\vre}$, where 
$(f_{(\pm,
  \delta)})_{\mathrm{main}}  = f_{(\pm,\delta)}(1-\chi_\varepsilon)$
and $(f_{(\pm,
  \delta)})_{\vre} =
f_{(\pm,\delta)}\chi_\varepsilon$,
%Let $\delta_n, \vre_n, \theta^{(n)}$ be
%positive sequences of parameters (which we will also optimize below), 
and let $$
f_{(\pm, \delta, \vre)} = f_{(\pm, \delta)} (1- \chi_{\eps}) -
I^\pm_{(\delta, \vre)}, \text{
 where } I^\pm_{(\delta, \vre)}= \int f_{(\pm, \delta)} (1- \chi_{\eps}) d\mu,
$$
i.e. $f_{(\pm, \delta, \vre)}$ is the projection of $\left(f_{(\pm,
  \delta)}\right)_{\mathrm{main}}$ to the space of zero integral
functions. 
% Inequality \equ{ineq:boundingmain} in the proof of \equ{eq: lem 1} shows that 
% \eq{eq: one estimate}{
% \|f_{(\pm, \delta, \vre)}\|_2^2=\|(f_{(\pm,
%   \delta)})_{\on{main}}\|^2_2-(I^\pm_{(\delta, \vre)})^2 \leq \|(f_{(\pm,
%   \delta)})_{\on{main}}\|^2_2  \ll 
% %\varepsilon_n^{-\ref{exp:boundbyell}}\max_r\int_{k\in K}|\psi_{(\pm,
% %  \delta)}(k_\theta re_1)|^2 dm_K\ll
% \varepsilon^{-\ref{exp:boundbyell}} \, \theta = \vre^{-\alpha_1}
% \delta^{1/2} .
% }
We note that 
%Using first the triangle inequality and then \equ{eq: lem 1} and
%\equ{eq: previous paragraph}, 
\eq{eq: another estimate}{
\cS_K\left(f_{(\pm, \delta, \vre)} \right) 
\leq  \cS_K \left(\left(f_{(\pm,
  \delta)}\right)_{\mathrm{main}} \right)+
\cS_K\left(I^\pm_{(\delta, \vre)} \right)
\ll
\varepsilon^{-\ref{exp:boundbyell}}\delta^{-1/2}+I^\pm_{(\delta,
    \vre)}.
}
Indeed, the first inequality follows from the triangle inequality, and
in the second inequality we used \equ{eq: lem 1} and
\equ{eq: previous paragraph} for the first summand and that fact that
$I^\pm_{(\delta, \vre)}$ is a constant.
By Siegel's formula, the term $I^\pm_{(\delta, \vre)}$ is uniformly
bounded (independently of $\vre, \delta$). 
% Note that the implicit constants in
% \equ{eq: another estimate} depend only on
% the support of $\psi_{(\pm, \delta)}$, which is uniformly bounded in measure. 

Having recorded these bounds, we turn to the application of Theorem
\ref{thm:ergodic}. We will choose a sequence $t_n \to \infty$, 
choose parameters $\lambda \in (0, \lambda_{\cL})$, set $\eta =
\frac{1}{\lambda+1}$, and for each $n$, define parameters $\delta_n,
\vre_n$, thus giving functions 
$$
f^\pm_n = f_{(\pm, \delta_n, \vre_n)}. 
$$
The theorem will be applied twice, to each of the two sequences $f^+_n,
f^-_n$. 
We will choose $0< \eta_1 < 2\eta$ so that \equ{eq:convergence
  condition} is satisfied. 
Then, since the $\nu_{t_n}(K)$ are
bounded below by $\abs{I}/4\pi$ for $t_n\ge t_{\abs{I}}$ (see (5.2)), using \equ{eq: another
  estimate} in Theorem \ref{thm:ergodic}
   % \equ{eq:convergence  condition}   %\equ{eq: ergodic estimate} 
  we obtain the
bound
\begin{equation}\label{eq:2summand}
\left| \pi_\cL(\Sigma_{\nu_{t_n}^\pm,t_n})f_{n}^\pm(\su) \right|
\ll 
e^{-(\eta-\frac{\eta_1}{2})  \lambda
  t_n}  D_n, 
\end{equation}
where 
\eq{eq: Dn}{
D_n \stackrel{\text{Def}}{=}\cS_K(f^\pm_n)=
\vre_n^{-\alpha_1}\delta_n^{-\frac{1}{2}}. 
}
In what follows we continue with the set of full measure
of $\su$ for which \eqref{eq:2summand} holds, and thus the 
implicit constants in the $\ll$ and $O(\cdot )$ notations may depend on $\su$. 
Let $I^\pm_n = I^\pm_{(\delta_n, \vre_n)}.$  
Since $f_{(\pm, \delta_n)}(1-\chi_{\eps_n})= f_n^\pm + I^\pm_n$, 
\equ{eq: lem 2} and \eqref{eq:2summand} imply
\begin{equation}
\label{eq:3summand}
\begin{split}
 \pi_\cL(\Sigma_{\nu_{t_n}^\pm,t_n}) & (f_{(\pm,
  \delta_n)} (1-\chi_{\eps_n}))(\su)  = \int_K I^\pm_n \nu^\pm_{t_n}
dm_K + O(e^{-(\eta-\frac{\eta_1}{2}) \lambda
  t_n}D_n) \\
= & \left(\int
f_{(\pm, \delta_n)} d\mu  \right) \;
\left(\int_K\nu_{t_n}^\pm dm_K\right)+
O(\varepsilon_n^{\beta}+e^{-(\eta-\frac{\eta_1}{2})\lambda
  t_n}D_n),
\end{split}
\end{equation}
where 
$$
\beta = \ref{exp:integrability}-\ref{exp:boundbyell}.
$$
Moreover \equ{eq: lem 3} implies 
\begin{equation}
\label{eq:1summand}
%\[
\pi_\cL(\Sigma_{\nu_{t_n}^\pm,t_n})(f_{(\pm,\delta_n)}\chi_{\eps_n})(\su)
\leq \pi_\cL(\Sigma_{t_n})(f_{(\pm,\delta_n)}\chi_{\eps_n})(\su)
\ll_{\su, \ref{exp:boundbyell},\ref{exp:integrability}}
\varepsilon_n^{\beta}.
\end{equation}
By Theorem~\ref{thm:siegelveech} we have $\int
f_{(\pm, \delta_n)} d\mu  = c  \int_{\bR^2}\psi_{(\pm,\delta_{n})}
dx,$
where $c = c(\cL, \cC)$ is the Siegel-Veech constant. 
Combining  this with \eqref{eq:3summand}  and \eqref{eq:1summand} we obtain
\begin{equation}\label{eq: final f}
%\begin{split}
%& 
\left | \pi_\cL(\Sigma_{\nu^\pm_{t_n},t_n})f_{(\pm,\delta_{n})}(\su)
  - 
%\pi_\cL(\Sigma_{\nu_{t_n},t_n})f_{\delta_{n}, \eps_{n}}(\su) %was this written like this before?
 c \, \left(\int_{\bR^2}\psi_{(\pm,\delta_{n})}
dx \right) \, \left(\int_K\nu_{t_n}^\pm dm_K\right) \right | 
%\\ 
\ll_\su
\varepsilon_n^{\beta}+e^{-(\eta
  - \frac{\eta_1}{2})\lambda  t_n}D_n.
%\end{split}
\end{equation}
Divide by $\tilde{\bf{\theta}}=e^{2t}\theta_t =\theta+O(\theta^{3}) =
\delta^{1/2}+O(\delta^{3/2})$ to find 
\begin{equation}\label{eq: final final f}
\left|\tilde{\bf{\theta}}^{-1}\pi_\cL(\Sigma_{\nu^\pm_{t_n},t_n})f_{(\pm,\delta_{n})}(\su)-c\tilde{|\textbf{I}|}\right|=\delta_n^{-1/2}
\left(\varepsilon_n^{\beta}+e^{-(\eta-\frac{\eta_1}{2})\lambda  
  t_n}D_n \right). 
\end{equation}
Combining \equ{eq:main term bound} and \eqref{eq: final final f}, and using, we get 
$$
\left| \tilde{\bf{\theta}}^{-1}
  \pi_\cL(\Sigma_{\nu^\pm_{t_n},t_n})f_{(\pm,\delta_{n})}(\su) - 
  \frac{c(\varphi_2-\varphi_1)}{2\pi} 
\right| \ll \delta_n^{1/2}+\delta_n^{-1/2}
\left(\varepsilon_n^{\beta}+e^{-(\eta-\frac{\eta_1}{2})\lambda 
  t_n}D_n\right). 
$$
Plugging this estimate into \equ{eq: what we claim2} and using
\equ{eq: Dn}, we find that for any $n$, 
\begin{equation}\label{finalinequal}
\left| \frac{N^{\cC}(e^{t_n},\su,\varphi_1,\varphi_2)}{e^{2t_n}} -
  \frac{c}{2}(\varphi_2-\varphi_1) 
\right| \ll
\delta_n^{1/2}+\delta_n^{-1/2}\varepsilon_n^{\beta}+ e^{-(\eta -
  \frac{\eta_1}{2})\lambda t_n} 
\vre_n^{-\alpha_1}\delta_n^{-1}.
\end{equation}

%\item
\medskip

\paragraph{\textbf{Step 3. Choosing parameters and deriving  bounds for any
    $T$. }}
Let $\lambda_\cL$ be the size of the spectral gap for $\cL, \mu$ as in
\S \ref{sec: four}, and let 
$\lambda < \lambda_\cL. $
We will show that
$\kappa = \frac{\lambda}{8(1+\lambda)} $ satisfies the conclusion of the 
theorem (see the
discussion following the statement of Theorem \ref{thm:theorem}). Let
%$\gamma=\frac{\sigma}{\lambda}$ 
%and let 
$t_n$ be the sequence defined by the equation 
\eq{eq: by the equation}{
e^{t_n} = n^{\frac{\sigma}{\lambda}} %\gamma.
}
where $\sigma>1$ will be chosen sufficiently large later (depending on $\lambda$).
Fix $1< \alpha_1 < \alpha_2 < 2$ and $\beta =
\alpha_2 - \alpha_1$ as in Lemma \ref{cutoff}, and set
\eq{eq: useful choice}{\eta_1=\frac{\alpha_1}{\sigma}
}
(note that with this choice, $0< \eta_1 < 2\eta$). 
This choice is necessary if we want
to satisfy \equ{eq:convergence condition}, since $\alpha_1>1$ implies
\[
\sum_{n\in \bN} e^{-\lambda \eta_1 t_n}=\sum_{n\in \bN}
n^{-\sigma\eta_1}=\sum_{n\in \bN} n^{-\alpha_1}<\infty.
\]

We bound the right hand side 
of \eqref{finalinequal}. First, we choose $\eps_n=\delta_n$, then the
first two terms are $O(\delta_n^{\beta-1/2})$ and the last term is
$O(n^{-\sigma(\eta-\frac{\eta_1}{2})}\delta_n^{-(1+\alpha_1)})$. To
optimize the asymptotics of these terms, we equalize
$\delta_n^{\beta-1/2}=n^{-A}\delta_n^{-(1+\alpha_1)}$, or
equivalently $\delta_n=n^{-\frac{A}{\alpha}}$ for
$A\stackrel{\text{Def}}{=}\sigma(\eta-\frac{\eta_1}{2})$ and $\alpha
\stackrel{\text{Def}}{=} 
\beta-\frac12+(1+\alpha_1) = \frac12 + \alpha_2$. 

The right hand side of \eqref{finalinequal} becomes (up to constants)
\eq{eq: for later use}{
n^{-(\beta-\frac12)\frac{A}{\alpha}}.
%$$
}
We let $\alpha_1\to 1$ and $\alpha_2\to 2$, so that
$(\beta-\frac12)\frac{A}{\alpha}$ is arbitrarily close 
to
\eq{eq: def B}{B \stackrel{\text{Def}}{=}
\left(1-\frac12 \right)\frac{(\sigma\eta-\frac12)}{1/2+2}=\frac{1}{5}\left(\sigma\eta-\frac12
  \right).
}
Thus for time squares $T_n^2 = e^{2t_n}=n^{\frac{2\sigma}{\lambda}}$,
\eqref{finalinequal} can be written as  
$$ \left|
N^{\cC}(T_n, \su, \varphi_1, \varphi_2) -
\frac{c}{2}(\varphi_2-\varphi_1)T_n^2 \right| \ll
T_n^{2\left(1- \kappa\right)}
$$
where
\eq{eq: def kappa}{
\kappa < \kappa(\sigma) \stackrel{\text{Def}}{=} \frac{\lambda}{2\sigma}B
}

Now for arbitrary $T$, let $n$ satisfy
$T_n <T\leq T_{n+1}$. By monotonocity, 
\eq{eq: monotonicity}{
\frac{c (\varphi_2-\varphi_1)}{2}  T_n^2
\left(1-O\left(T_n^{-2\kappa} \right)  \right) \leq
N^\cC(T,\su,\varphi_1,\varphi_2) \leq \frac{c
  (\varphi_2-\varphi_1)}{2} T_{n+1}^2 \left( 
1+O\left(T_{n+1}^{-2\kappa} \right) 
%1+ O(T_{n+1}^{-\frac{\lambda}{\sigma}} ) 
\right).
}

Since $e^{t_n}= n^{\frac{\sigma}{\lambda}}
%n^\gamma
$ we have
$$\max \left(\frac{T^2}{T^2_n}, \frac{T^2_{n+1}}{T^2} \right) \ll
\frac{T_{n+1}^2}{T_n^2} =\left(1+\frac{1}{n} \right)^{\frac{2 \sigma}{\lambda} 
%\gamma
} =
1+O\left(\frac{1}{n}\right).
$$ 
So both sides of \equ{eq: monotonicity} are
$\frac{c(\varphi_2-\varphi_1)}{2} 
T^2\left(1+O\left(T^{-\frac{\lambda}{2\sigma}} \right)^2+ O
  \left(T^{-\kappa} \right)^2\right).$  
Thus we can do no better than to set $\kappa(\sigma)  =
\frac{\lambda}{2\sigma},$ i.e., by \equ{eq: def kappa}, $B= 1$ which
leads via \equ{eq: def B} and \equ{eq: def kappa} to $\sigma = \frac{5.5}{\eta}$. 
% $\kappa=\frac{\lambda\eta}{8}-\frac{\lambda}{16\sigma}$.
We may choose
$\sigma$ large and increase our original choice of
$\lambda<\lambda_\cL$, to see that we can take
$\kappa=\frac{\lambda}{11(\lambda+1)}$. We leave it to the reader
to verify that with 
any choice of $\sigma$, 
\equ{eq: one more requirement} is satisfied.
%\end{proof}
\qed

\section{Effective counting in all sectors and all dilates of an
  ellipse}
In order to change the order of quantifiers and obtain an estimate
simultaneously true for all $\varphi_1, \varphi_2$ and all $\{g \su : g
\in G\}$, we will use two distinct techniques. Firstly we will use
Theorem \ref{thm: uniform ellipses} instead of \ref{thm:ergodic}, as
this will allow us to control countably many ellipses. Secondly we
will give an additional approximation argument which shows how to use
countably many functions, approximating a countable
dense set of sectors,  along with a
countable dense set of ellipses, of countably many radii, to simultaneously
control all ellipses and all sectors. We proceed to the details.

\begin{proof}[Proof of Theorem \ref{thm: strengthening}]
We will use the notations and estimates as in the proof of
Theorem \ref{thm:theorem}, and follow the same steps. We fix a
configuration $\cC$ and use it 
throughout, and let $c$ be the corresponding Siegel-Veech constant. In
analogy with \equ{eq: in analogy}, for fixed 
$\varphi_1, \varphi_2$ we set 
$$I = [\varphi_1, \varphi_2], \ \ 
I^-_{t, \varphi_1, \varphi_2} = [\varphi_1+\theta_t,\varphi_2-\theta_t], \ \
I^+_{t, \varphi_1, \varphi_2} =
[\varphi_1-\theta_t,\varphi_2+\theta_t],
$$
and let 
$
\nu, \, \nu^-_{t, \varphi_1, \varphi_2}, \, \nu^+_{t, \varphi_1, \varphi_2}
$
denote respectively the indicator functions of these
intervals (where we have selected a different notation to reflect the dependence on
$\varphi_1, \varphi_2$). 
%For a function $f$ defined on a $G$-space, we denote by $f
%\circ g$ the function $x \mapsto f(gx)$. 
Using \equ{eq: what we claim}, and using
that $\pi_\cL\left(\Sigma_{\nu^\pm_t, \varphi_1, \varphi_2} \right)(g \su) =
\pi_\cL \left(\Sigma^{(g)}_{\nu^\pm_t, \varphi_1, \varphi_2} \right)(\su)$,  
%$\widehat{\mathbbm{1}_W}(g \su) =
%\widehat{\mathbbm{1}_{W}} \circ g(\su) = \left(\mathbbm{1}_W \circ
 % g \right) \widehat (\su)$, 
we find that for every $\su \in 
\cH$ and every $g \in G$,
%\eq{eq: step1}{
$$
\pi_{\cL}\left(\Sigma^{(g)}_{\nu^-_{t, \varphi_1, \varphi_2},
  t} \right)\widehat{\mathbbm{1}_{W_1(\theta)} } (\su) 
\leq \frac{\theta_t}{\pi}
N^{\cC}(e^t,g\su,\varphi_1,\varphi_2) \leq 
\pi_{\cL}\left(\Sigma^{(g)}_{\nu^+_{t, \varphi_1, \varphi_2},
  t} \right)\widehat{\mathbbm{1}_{W_2(\theta)}} (\su).  
%}
$$
This estimate %\equ{eq: step1}  
generalizes \equ{eq:
  what we claim} and constitutes the first step of the proof. 
\ignore{
In order to give smooth approximations to
$\widehat{\mathbbm{1}_{W} \circ g}$ we will no longer be able to use the
functions $\psi_{(\pm , \delta)}$ in \equ{eq: def psi}, since the
compositions of these functions with $g \in G$ is not
$K$-smooth. Hence we define 
$$
\Psi_{(-, \delta, \varphi_1, \varphi_2)}(r \cos \beta, r \sin \beta) =
H_1(\beta) H_2(r),   
$$
where $H_1$ and $H_2$ are continuously differentiable, with
derivatives of size $O(\delta^{-1})$, and such that $H_1$ is 1 on
$[\varphi_1 - \delta, \varphi_2+\delta]$ and vanishes outside
$[\varphi_1, \varphi_2]$, and $H_2$ is 1 on $[0, 1-\delta]$
and vanishes outside $[0, 1+\delta]$. Then for any $g
\in G$, we set 
$$
\Psi_{(-, \delta, \varphi_1, \varphi_2, g)} = \Psi_{(-, \delta,
  \varphi_1, \varphi_2)} \circ g,
$$
so that 
\eq{eq: any g}{\Psi_{(-, \delta, \varphi_1, \varphi_2, g)}  \leq
\mathbbm{1}_{W_1(\theta)} \circ g \  \text{ and } \|\partial_\theta \Psi_{(-,
  \delta, \varphi_1, \varphi_2, g)}\|_\infty \ll \|g\| \delta^{-1}}
(where $\|g\|$ denotes the operator norm of $g$ as a map $\R^2 \to
\R^2$). Similarly we define $\Psi_{(+, \delta, \varphi_1, \varphi_2, g)}$
satisfying
\eq{eq: any g 2}{
\mathbbm{1}_{W_2(\theta)} \circ g\leq 
\Psi_{(+, \delta, \varphi_1, \varphi_2, g)}  
\  \text{ and } \|\partial_\theta \Psi_{(+,
  \delta, \varphi_1, \varphi_2, g)} \| \ll \|g\| \delta^{-1},
}
and as in \eqref{ineq:approxbysmooth}, obtain bounds
\eq{ineq:approxbysmooth2}{
\int_{\bR^2}(\mathbbm{1}_{W_1(\theta)} \circ g-\Psi_{(-,\delta,
  \varphi_1, \varphi_2, g)}) dx\ll \delta \; \text{ and }
\int_{\bR^2}(\Psi_{(+,\delta, \varphi_1, \varphi_2,g )} -
\mathbbm{1}_{W_2(\theta)} \circ g) dx\ll \delta. 
}
The supports of these functions satisfy 
\eq{eq: support}{
\supp \, \Psi_{(\pm, \delta, \varphi_1, \varphi_2, g)} \subset B(0, 2\|g\|)
}
independently of $\delta, \varphi_1, \varphi_2$.

} 

In the second step of the proof
we again need to record certain
bounds, but this time we will record their
dependence on three additional parameters. Namely, as
before, we will have parameters $1 < \alpha_1 < \alpha_2 <2, \, \beta
= \alpha_2-\alpha_1, \, 
\eta_1>0, \, $ as well as sequences of times $t_n \nearrow \infty$,
smoothing parameters $\delta_n$ and cutoff parameters
$\vre_n$. In addition we will have sequences of `ellipse parameters'
$(g_n) \subset G$ and `angular sector parameters'
$\varphi_1^{(n)}< \varphi_2^{(n)}$ with $\varphi_2^{(n)} -
\varphi_1^{(n)} \leq 2\pi$. 
%We will derive a bound for 
%$
%\left|N^{\cC} \left(e^{t_n}, g_n \su, \varphi_1^{(n)}, \varphi_2^{(n)}\right)
%  -\frac{c}{2} \left(\varphi_2^{(n)} - \varphi_1^{(n)}\right)\right|  
%$
%which is analogous to \eqref{finalinequal}.

Using a smoothing parameter 
$\delta = \delta_n$, defining the functions 
$f_{(\pm, \delta)}$ (Siegel-Veech transforms of smooth approximations of
$\mathbbm{1}_{W_1}, \mathbbm{1}_{W_2}$) as before, and in analogy with
\equ{eq: what we claim2}, we obtain
%\eq{eq: main estimate ellipses}{
$$
\frac{\pi}{\theta_t} \pi_{\cL} \left(\Sigma^{(g)}_{\nu^-_{t,
      \varphi_1, \varphi_2}}\right) f_{(-, \delta)}(\su) 
\leq 
N^{\cC}(e^t, g\su, \varphi_1, \varphi_2)
\leq
\frac{\pi}{\theta_t} \pi_{\cL} \left(\Sigma^{(g)}_{\nu^+_{t,
      \varphi_1, \varphi_2}}\right) f_{(+, \delta)}(\su). 
%}
$$
Note that these upper and lower bounds are valid for any $g \in G$ and
any $\varphi_1, \varphi_2$. 

In the proof of \eqref{finalinequal}, there are two sources for the
dependence of the  estimate on
$\su$. The first arises in deriving \eqref{eq:1summand} by way of
\equ{eq: lem 3}, and gives rise to an estimate which is uniform as
$\su$ ranges over a compact subset of $\cH$, and the second arises from Theorem
\ref{thm:ergodic}, and gives rise to a condition $n \geq n_0(\su)$. 
Thus the same argument (with Theorem \ref{thm:
  uniform ellipses} instead
of Theorem \ref{thm:ergodic}) gives the generalization of \eqref{finalinequal},
%$$
\eq{eq: as before}{
\begin{split}
& \left|
  \frac{N^{\cC}\left(e^{t_n},g_n\su,\varphi_1^{(n)},\varphi_2^{(n)}\right)}{e^{2t_n}}
  - \frac{c}{2}\left(\varphi^{(n)}_2-\varphi^{(n)}_1\right) 
\right|  \\ \ll &
\delta_n^{1/2}+\delta_n^{-1/2}\varepsilon_n^{\beta}+ \left(\varphi^{(n)}_2-\varphi^{(n)}_1\right)^{1-\lambda\eta/2} e^{-(\eta -
  \frac{\eta_1}{2})\lambda t_n} 
\vre_n^{-\alpha_1}\delta_n^{-1},
\end{split}
}
%$$
as long as $n \geq n_0(\su)$ and where the implicit constant depends
on $g_n$ and $\su$ and can be 
taken to be uniform in compact subsets of $G$ and $\cL$. This
completes the second step of the proof.

We now choose $\lambda, 
\alpha_1, \alpha_2$ satisfying 
$
\lambda < \lambda_{\cL}, \ 1 < \alpha_1 < \alpha_2 < 2. 
$
For each $n \in \N$, we define an auxiliary variable 
$$
m=m_n = \lfloor n^{1/7} \rfloor, 
$$
which we will refer to as the {\em scale} of $n$. For fixed $m$, let 
$$\mathcal{N}_m = \{n: m_n = m\}$$ 
denote the indices of scale $m$. Note that as $n \to
\infty$, the scales $m_n$ also tend to infinity at a slower rate, and
the cardinality of $\mathcal{N}_m$ is approximately $m^6.$ 
Now choose $\varphi^{(n)}_1, 
\varphi_2^{(n)}, g_n$ so that for all large enough $m$, the collection of triples
$$
\left\{\left(\varphi_1^{(n)}, \varphi_2^{(n)}, g_n \right): n \in
  \mathcal{N}_m \right\}
$$
is $\frac{1}{m (\log m)}$-dense in 
\eq{eq: need to fill}{\left\{ \left(\varphi_1, \varphi_2, g \right):
    \varphi_1 \in [0, 2\pi], \ \varphi_2 - 
\varphi_1 \in [0, 2\pi], \ g \in G, \ \max (\|g\|, \|g^{-1}\|)
< \log m \right\}}
(with respect to the sup-norm in the first two coordinates and the
operator norm in the third coordinate), and so that 
$\varphi_2^{(n)} - \varphi_1^{(n)} \geq \frac{1}{m \log m}$. 
This is possible since \equ{eq: need to fill} defines a 5-dimensional 
manifold of diameter $O(\log m)$.

Now following \equ{eq: by the equation} we choose $t_n$ so that $e^{t_n} =
m_n^{\frac{\sigma}{\lambda}},$ where $\sigma$ is a parameter we will
optimize. The optimal value will turn out to be 
\eq{eq: result for sigma}{\sigma = 8.5(\lambda +1),}
assume for now it is large.
 Let 
\eq{eq: useful choice2}{
\eta_1 = \frac{7\alpha_1}{\sigma} > \frac{7}{\sigma},
}
so that
%, and let 
%$$ \, \delta_n= \vre_n = e^{-\frac{\lambda}{8} t_{n}}, \,
%\theta^{(n)}=\delta_n^{1/2}. $$
%With these parameters, using \equ{eq: bound on eta1} we have 
$$
\sum_{n \in \N} e^{-\lambda \eta_1 t_n } %= \sum_{n \in \N}
%\left(e^{ t_n} \right)^{-\lambda \eta_1}
= \sum_{m
  \in \N} \sum_{n \in \mathcal{N}_m} 
m^{-\sigma \eta_1} \ll
\sum_{m \in \N} m^{6-\sigma \eta_1}< \infty, 
$$
so that \equ{eq:convergence condition} holds. 
Also note that the lengths of intervals at scale $m_n$ is bounded
below by $e^{-\frac{\lambda t_n}{\sigma}}/t_n$ and in
  particular, since  $\frac{\lambda}{\sigma}<2$, satisfies the lower
  bound $|I|>e^{-2t_n}$ for all large enough $n$. Thus we can apply 
  Theorem \ref{thm: uniform ellipses}, and deduce
\equ{eq: as before}.

As before, we optimize the right hand side of \equ{eq: as before} by setting
all three summands equal to each other, and we obtain that it is
bounded by a constant (depending on $\|g_n\|$) multiplied by the expression
\equ{eq: for later use}. Letting $\alpha\to 2, \beta \to 1$ in \equ{eq: for
  later use} and using \equ{eq: useful choice2} instead of \equ{eq: useful choice}, we find that the
right hand side of \equ{eq: as before} is on the order of 
$
n^{-\frac15(\sigma\eta-\frac72)}=\left(e^{2t_n} \right)^{-\kappa} %=
                                %(m_n)^{-2\frac{\sigma}{\lambda} 
                                %\kappa}
$ where 
$$\kappa = \kappa(\sigma) =
\frac{\lambda}{10}\left(\eta-\frac{7}{2\sigma} \right). 
$$
Denote
$$
S(r_0, \varphi_1, \varphi_2) = \{r (\cos \beta, \sin \beta): \beta \in
[\varphi_1, \varphi_2], 0 \leq r \leq r_0\}.
$$
Let $g \in G$ and $\varphi_1 \in [0, 2\pi]$ and $\varphi_2 \in \R$ with
$\varphi_2-\varphi_1\leq 2\pi$. When $\|g_j-g\|< \frac{1}{m(\log m)}$,
and $ g, g_j$ are as in \equ{eq: need to fill}, then $\max(\|\mathrm{Id} -
g_jg^{-1}\|, \|\mathrm{Id} - gg_j^{-1}\|)< \frac{1}{m}$. Thus there is a constant $c_1$
such that for all large enough $m$ there are
$k, \ell \in \mathcal{N}_m$ such that 
$$\varphi_1^{(\ell)} < \varphi_1 < \varphi_1^{(k)} < \varphi_2^{(k)} <
\varphi_2 < \varphi_2^{(\ell)}, $$
$$
\left| \varphi_i^{(\ell)} - \varphi_i^{(k)} \right| < \frac{c_1}{m}, \
i=1,2, 
$$
and for any $r_0$ we have the inclusions 
$$
g_k^{-1} S\left(r_0\left(1-\frac{c_1}{m}\right), \varphi^{(k)}_1, \varphi^{(k)}_2 \right)\subset g^{-1} S\left(r_0,
  \varphi_1, \varphi_2 \right) \subset g_\ell^{-1} S
\left(r_0\left(1+\frac{c_1}{m} \right),
  \varphi^{(\ell)}_1, \varphi^{(\ell)}_2 \right). 
$$
Hence for all $T$, 
$$
N^{\cC}\left(T\left(1 -
    \frac{c_1}{m}\right),g_k\su,\varphi_1^{(k)},\varphi_2^{(k)}\right) 
\leq 
N^{\cC}\left(T,g\su,\varphi_1,\varphi_2\right) \leq
N^{\cC}\left(T\left( 1+\frac{c_1}{m}
\right),g_\ell\su,\varphi_1^{(\ell)},\varphi_2^{(\ell)}\right). 
$$
Choosing $n$ so that $e^{t_n} \leq T < e^{t_{n+1}}$, and assuming $T$
and hence $m$ are 
 large enough so that the preceding estimates are all satisfied,
arguing as in the preceding proof, we obtain the following analogue of 
\equ{eq: monotonicity}:
\begin{equation} \label{eq: monotonicity 2} \begin{split}
& \left(\varphi_2^{(\ell)}-\varphi_1^{(\ell)} \right) \frac{c}{2} e^{2t_n} \left( 1-
  O((e^{2t_n})^{-\kappa}) \right) \left( 1 - O\left(\frac{1}{m}
  \right)\right)\\ \leq &  N^{\cC}(T, g\su, \varphi_1, 
\varphi_2) \\  \leq & \left(\varphi_2^{(k)}-\varphi_1^{(k)} \right) \frac{c}{2} e^{2t_{n+1}} \left( 1+
  O((e^{2t_{n+1}})^{-\kappa}) \right) \left( 1 + O\left(\frac{1}{m}
  \right) \right)
\end{split} \end{equation}
(with implicit constants depending on $\|g\|$). As before 
$$
\frac{e^{2t_{n+1}}}{e^{2t_n}} = 1+O\left(T^{2(-\frac{\lambda}{2\sigma})} \right),
$$
and since 
$$\max \left[ \frac{\varphi_2-\varphi_1}
{\varphi_2^{(k)}-\varphi_1^{(k)} } 
, \, 
\frac{\varphi_2^{(\ell)}-\varphi_1^{(\ell)}
}{\varphi_2-\varphi_1 } 
\right ] = 1+O\left(\frac{1}{m} \right) =
1+O\left((e^{2t_n})^{\frac{-\lambda}{2\sigma}} \right),$$
both sides of \eqref{eq: monotonicity 2} are  $(\varphi_2-\varphi_1)
\frac{c}{2} T^2\left(1+O\left(T^{-\kappa'} \right)\right), $ where 
$$
\kappa' = \min \left\{\kappa(\sigma), \frac{\lambda}{2\sigma} \right\} .
$$
Setting both of these terms equal to each other and computing $\sigma$
gives \equ{eq: result for sigma}. When we plug this in we get the
required estimate, with
$$ \kappa = \frac{\lambda}{17(\lambda+1)},
$$
completing the proof.
\end{proof}

\ignore{
\section{Everywhere Counting on Lattice Surfaces}

In \cite{EsMcMu} a method of counting on affine symmetric spaces
$H\backslash G$ - simplifying an argument of \cite{DRS} - is
introduced using mixing for on $\hom$. While $G=\on{SL}_2(\bR)$ and
$U$ the subgroup of (upper) unipotent matrices does not fall into that
regime, a proof in this settings is sketched in the last section of
this paper. It is also noted there, that the method is most likely to
be made quantitative replacing mixing with effective mixing. Indeed,
this has been written out in \cite{BeOh}, but again excluding the case
of $U$ unipotent. Since the method is well known, and very robust, we
only outline the proof here. 
}

\bibliographystyle{alpha}

\end{document}